\documentclass[12pt,reqno]{amsart}

\usepackage{subcaption}
\usepackage{amssymb,bbm,units,url}
\usepackage{booktabs}
\usepackage[bookmarks=false,unicode,colorlinks,urlcolor=red,citecolor=red,linkcolor=blue]{hyperref}
\usepackage{aliascnt}
\usepackage{doi}
\usepackage{tikz}
\usepackage{pgfplots} 

\usepackage{color}

\usepackage[margin=1.25in]{geometry}

\newtheorem{theorem}{Theorem}

\newaliascnt{lemma}{theorem}
\newtheorem{lemma}[lemma]{Lemma}
\aliascntresetthe{lemma}

\newaliascnt{proposition}{theorem}
\newtheorem{proposition}[proposition]{Proposition}
\aliascntresetthe{proposition}

\newaliascnt{corollary}{theorem}
\newtheorem{corollary}[corollary]{Corollary}
\aliascntresetthe{corollary}

\newaliascnt{conjecture}{theorem}

\aliascntresetthe{conjecture}

\theoremstyle{definition}

\newaliascnt{example}{theorem}
\newtheorem{example}[example]{Example}
\aliascntresetthe{example}

\newtheorem*{remark}{Remark}

\newtheorem{cond}{Condition}[section]
\newtheorem{paramconcave}[cond]{Parameter Assumption}
\newtheorem{paramconvex}[cond]{Parameter Assumption}
\newtheorem{concavecond}[cond]{Concave Condition}
\newtheorem{convexcond}[cond]{Convex Condition}
\newtheorem{concavecondweak}[cond]{Weaker Concave Condition}
\newtheorem{convexcondweak}[cond]{Weaker Convex Condition}

\makeatletter
\def\tagform@#1{\maketag@@@{\ignorespaces#1\unskip\@@italiccorr}}
\let\orgtheequation\theequation
\def\theequation{(\orgtheequation)}
\makeatother
\def\equationautorefname~{}

\newcommand{\arxiv}[1]{%
 \href{http://front.math.ucdavis.edu/#1}{ArXiv:#1}}

\newcommand{\link}{%
\href{https://www.dropbox.com/s/z0tlkypbm4xknlx/Lattice_Counting.pdf?dl=0}{www.math.illinois.edu/~laugesen/}}

\DeclareMathOperator*{\argmax}{\arg\!\max}

\newcommand{\e}{\varepsilon}

\newcommand{\N}{{\mathbb N}}
\newcommand{\R}{{\mathbb R}}

\newcommand{\area}{\operatorname{Area}}
\newcommand{\ud}{\,\mathrm{d}}
\newcommand{\sN}{\widetilde{N}}
\newcommand{\sGamma}{\widetilde{\Gamma}}
\newcommand{\se}{{\mathcal E}}
\newcommand{\sL}{\widetilde{L}}
\newcommand{\sM}{\widetilde{M}}

\newcommand{\sO}{\widetilde{O}}

\newcommand{\shiftf}{\widetilde{f}}
\newcommand{\sg}{\widetilde{g}}
\newcommand{\sa}{\widetilde{\alpha}}

\begin{document}

\title{Shifted lattices and asymptotically optimal ellipses}
\author[]{Richard S. Laugesen and Shiya Liu}
\address{Department of Mathematics, University of Illinois, Urbana,
IL 61801, U.S.A.}
\email{\ Laugesen\@@illinois.edu \quad sliu63\@@illinois.edu}
\date{\today}

\keywords{Translated lattice, concave curve, convex curve, $p$-ellipse, spectral
optimization, Dirichlet Laplacian, Schr\"odinger eigenvalues, harmonic oscillator.}
\subjclass[2010]{\text{Primary 35P15. Secondary 11P21, 52C05}}

\begin{abstract}
Translate the positive-integer lattice points in the first quadrant by some amount in the horizontal and vertical directions. Take a decreasing concave (or convex) curve in the first quadrant and construct a family of curves by rescaling in the coordinate directions while preserving area. Consider the curve in the family that encloses the greatest number of the shifted lattice points: we seek to identify the limiting shape of this maximizing curve as the area is scaled up towards infinity. 

The limiting shape is shown to depend explicitly on the lattice shift. The result holds for all positive shifts, and for negative shifts satisfying a certain condition. When the shift becomes too negative, the optimal curve no longer converges to a limiting shape, and instead we show it degenerates.

Our results handle the $p$-circle $x^p+y^p=1$ when $p>1$ (concave) and also when $0<p<1$ (convex). Rescaling the $p$-circle generates the family of $p$-ellipses, and so in particular we identify the asymptotically optimal $p$-ellipses associated with shifted integer lattices. 

The circular case $p=2$ with shift $-1/2$ corresponds to minimizing high eigenvalues in a symmetry class for the Laplacian on rectangles, while the straight line case ($p=1$) generates an open problem about minimizing high eigenvalues of quantum harmonic oscillators with normalized parabolic potentials.  

\end{abstract}

\maketitle

\section{\bf Introduction}

Among all ellipses centered at the origin with given area, consider the one enclosing the maximum number of positive integer lattice points. Does it approach a circular shape as the area tends to infinity? Antunes and Freitas \cite{AF13}  showed the answer is yes. We tackle a variant of the problem in which the lattice is translated by some increments in the $x$- and $y$-directions, and show the asymptotically optimal ellipse is no longer a circle but an ellipse whose semi-axis ratio depends explicitly on the translation increments. This optimal ratio succeeds in ``balancing'' the horizontal and vertical empty strip areas created by the translation of the lattice; see \autoref{fig:opt_ellipse}. The precise statement is given in \autoref{th:S_limit_shift}.
\begin{figure}
\includegraphics[scale=.4]{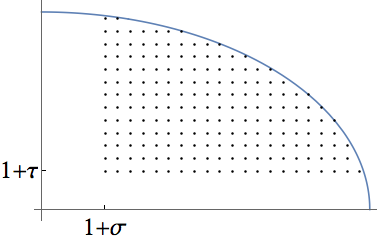}
\caption{\label{fig:opt_ellipse}An ellipse that maximizes (among ellipses with the same area) the number of enclosed positive-integer lattice points shifted by $4$ units horizontally and $2$ units vertically. This optimal ellipse roughly balances the areas of the  horizontal and vertical empty strips; see \autoref{th:S_limit_shift}.}
\end{figure}

Generalized ellipses obtained by stretching a (concave or convex) smooth curve can be handled by our methods too, in \autoref{th:S_limit_shift_general}. The results hold for all positive translations, and for small negative translations that satisfy a computable, curve-dependent criterion. 

When the curve is a straight line, one arrives at an open problem for right triangles that contain the most lattice points. The shape of these triangles exhibits a surprising clustering behavior as the area tends to infinity, as  revealed by our numerical investigations in  \autoref{sec:numerics}. This clustering conjecture has been investigated recently in the unshifted case by Marshall and Steinerberger \cite{marshall_steinerberger}.

\autoref{section_spectral} motivates this paper by connecting to spectral minimization problems for the Dirichlet Laplacian, and raises conjectures for the quantum harmonic oscillator and for a whole family of such Schr\"{o}dinger eigenvalue problems. The recent advances on high eigenvalue minimization  began with work of Antunes and Freitas \cite{AF12,AF13,AF16,f16}, and continued with contributions from van den Berg, Bucur and Gittins \cite{BBG16a}, van den Berg and Gittins \cite{BBG16b}, Bucur and Freitas \cite{BF13}, Gittins and Larson \cite{GL17}, Larson \cite{Larson}, and Marshall \cite{marshall}. 

We show in \autoref{section_spectral} that the original result of Antunes and Freitas does not extend to the subclass of symmetric eigenvalues. Instead, the optimal rectangle degenerates in the limit. 

\begin{remark}
The lattice point counting estimates in this paper are similar to those used for the Gauss circle problem, which aims for accurate asymptotics on the counting function inside the circle (and other closed curves) as the area grows to infinity. The best known error estimate on the Circle Problem is due to Huxley \cite{Hux03}. 

The lattice counting formulas in our paper differ somewhat from that work, because we consider only one quadrant of lattice points and our regions contain empty strips due to the translation of the lattice. Further, we focus on proving suitable inequalities (rather than asymptotics) for the counting function, in order to prevent the maximizing shape from degenerating. In essence, we develop inequalities that trade off the empty regions in the vertical and horizontal directions. After degeneration has been ruled out, we can invoke asymptotic formulas with error terms that need not be as good as Huxley's in order to prove convergence to a limiting shape.
\end{remark} 

\section{\bf  Results}\label{section:results}

Consider a strictly decreasing curve $\Gamma$ lying in the first quadrant with $x$- and $y$-intercepts at $L$. Represent the curve as the graph of $y= f(x)$ where $f$ is strictly decreasing for $x\in[0,L]$. Denote the inverse function of $f$ as $g(y)$ for $y \in [0,L]$. Now compress the curve by a factor $s>0$ in the $x$-direction and stretch it by the same factor in the $y$-direction: 
 \begin{equation*}
 \Gamma(s) = \text{graph of } sf(sx).
 \end{equation*} 
Next scale the curve $\Gamma(s)$ by a factor $r>0$:
\begin{equation*}
r\Gamma(s)= \text{graph of } rsf(sx/r) .
\end{equation*} 

Given numbers $\sigma, \tau>-1$, consider the translated or \emph{shifted} positive-integer lattice
\[
(\N+\sigma) \times (\N+\tau) ,
\]
which lies in the open first quadrant. Define the shifted-lattice counting function under the curve $s\Gamma(s)$ to be
\begin{align*}
N(r,s) &=\text{number of shifted positive-integer lattice points lying inside or on $r\Gamma(s)$ } \\
&=\# \big\{ (j,k)\in\mathbb{N} \times \mathbb{N}:k+\tau\leq rsf \big( (j+\sigma)s/r \big) \big\}. 
\end{align*}
The set $S(r)$ consists of $s$-values that maximize $N(r,s)$, that is, 
\[
S(r)= \argmax_{s>0} N(r,s), \qquad r>0 .
\] 

Write
\[
x^- = 
\begin{cases}
0 , & x \geq 0, \\
|x| , & x < 0
\end{cases} .
\]

Our first theorem will say that the maximizing set $S(r)$ is bounded, under either of the following conditions on the shift parameters $\sigma,\tau>-1$. 
\begin{paramconcave}\label{pa:concave}
$\Gamma$ is concave and strictly decreasing, with
\begin{equation}\label{eq:concave_f_condition_shift}
\max\Big\{f\big(\frac{1-\sigma^-}{2-\sigma^-}L\big),g\big(\frac{1-\tau^-}{2-\tau^-}L\big)\Big\} <2\Big(\frac{1}{2}-\sigma^--\tau^-\Big)L. 
\end{equation} 
\end{paramconcave}
\begin{paramconvex}\label{pa:convex}
$\Gamma$ is convex and strictly decreasing, with 
\begin{equation}\label{eq:convex_sigma}
\min \Big\{(1-\sigma^-)f\big(\frac{1-\sigma^-}{2-\sigma^-}L\big), (1-\tau^-)g\big(\frac{1-\tau^-}{2-\tau^-}L\big)\Big\} > 2(\sigma^- + \tau^-) L
\end{equation}
and
\begin{align}
\mu_f(\sigma) & \overset{\text{def}}{=} \min \Big\{ (1+\sigma)f\big(\frac{1+\sigma}{2+\sigma}x\big)-f(x) : \frac{1+\sigma}{2+\sigma}L \leq x \leq L\Big\} > 0, \label{eq:mu_f} \\
\mu_g(\tau) & \overset{\text{def}}{=} \min \Big\{ (1+\tau)g\big(\frac{1+\tau}{2+\tau}y\big)-g(y) : \frac{1+\tau}{2+\tau}L \leq y \leq L\Big\} > 0. \label{eq:mu_g}
\end{align}
\end{paramconvex}
When $\sigma,\tau \geq 0$,  
conditions~\autoref{eq:concave_f_condition_shift} and \autoref{eq:convex_sigma} hold automatically (using that $0<f(x)<L$ and $0<g(y)<L$ when $x,y \in (0,L)$) and conditions \autoref{eq:mu_f} and \autoref{eq:mu_g} also hold (using that $f$ and $g$ are strictly decreasing and positive). Thus the Parameter Assumptions are significant only when $\sigma<0$ or $\tau<0$. They are used to obtain an upper bound on the counting function: see the comments after \autoref{prop:ineq_shift} and \autoref{prop:ineq_convex_shift}. 
\begin{theorem}[Uniform bound on optimal stretch factors] \label{thm:s_bounded_shift}
If the curve $\Gamma$ and the shift parameters $\sigma,\tau> -1/2$ satisfy Parameter Assumption~\autoref{pa:concave} or \autoref{pa:convex}, then for each $\e>0$ one has 
\[
S(r) \subset \big[B(\tau,\sigma)^{-1}-\e,B(\sigma,\tau)+\e \big] \qquad \text{for all large $r$,}
\]
where 
\[
B(\sigma,\tau)= \frac{2+\sigma+\tau + \sqrt{ (2+\sigma + \tau)^2 - 4(\sigma+1/2)\tau )}}{2(\sigma+1/2)}.
\]
\end{theorem}
The bounding constant $B(\sigma,\tau)$ depends only on the shift parameters, not on the curve $\Gamma$. The bounding constant $B(0,0)=4$ in the unshifted case is consistent with our earlier work \cite[Theorem~2]{Lau_Liu17}. 

\autoref{thm:s_bounded_shift} is proved in \autoref{section_proof_bounded}. Note it does not assume the curve is smooth. 

If the curve is smooth, then the optimal stretch set $S(r)$ for maximizing the lattice count is not only bounded but converges asymptotically to a computable value, as stated in the next theorem. First we state the smoothness conditions to be used. 
\begin{concavecond} \label{cond:concave}
$\Gamma$ is concave, and for some $(\alpha ,\beta) \in \Gamma$ with $\alpha,\beta>0$ one has $f\in C^2[0,\alpha], g \in C^2[0,\beta]$, with
\begin{align*}
& \text{$f'<0$ on $(0,\alpha]$, $f''<0$ on $[0,\alpha]$, $f''$ monotonic on $[0, \alpha]$,} \\
& \text{$g'<0$ on $(0,\beta]$, $g''<0$ on $[0,\beta]$, $g''$ monotonic on $[0, \beta]$.}
\end{align*}
\end{concavecond}
\begin{convexcond} \label{cond:convex}
$\Gamma$ is convex, and for some $(\alpha ,\beta) \in \Gamma$ with $\alpha,\beta>0$ one has $f\in C^2[\alpha,L], g \in C^2[\beta,L]$, with
\begin{align*}
& \text{$f'<0$ on $[\alpha,L)$, $f''>0$ on $[\alpha,L]$, $f''$ monotonic on $[\alpha,L]$,} \\
& \text{$g'<0$ on $[\beta,L)$, $g''>0$ on $[\beta,L]$, $g''$ monotonic on $[\beta,L]$.}
\end{align*}
\end{convexcond}
\begin{theorem}[Sufficient conditions for asymptotic balance of optimal curve]\label{th:S_limit_shift} If the curve $\Gamma$ and shift parameters $\sigma, \tau>-1/2 $ satisfy either Parameter Assumption~\autoref{pa:concave} and Concave Condition~\ref{cond:concave}, or Parameter Assumption~\autoref{pa:convex} and Convex Condition~\ref{cond:convex}, then the stretch factors maximizing $N(r,s)$ approach
\[
s^*=\sqrt{\frac{\tau + 1/2}{\sigma+1/2}}
\]
as $r \to \infty$, with 
\[
S(r) \subset [s^* -O(r^{-1/6}), s^* + O(r^{-1/6})],
\]
and the maximal lattice count has asymptotic formula
\[
\max_{s>0} N(r, s) = r^2 \area(\Gamma)-2rL \sqrt{(\sigma+1/2)(\tau+1/2)} + O(r^{2/3}).
\]
In particular, when the shift parameters $\sigma$ and $\tau$ are equal, the optimal stretch factors for maximizing $N(r,s)$ approach $s^*=1$ as $r \to \infty$.
\end{theorem}
The theorem follows from \autoref{th:S_limit_shift_general} below, which has weaker hypotheses. 

We call the optimally stretched curve ($s=s^*$) ``asymptotically balanced'' in terms of the shift parameters, because the optimal shape balances the areas of the empty strips that are created by translation of the lattice: a horizontal rectangle of width $rL/s^*$ and height $\tau+1/2$ has the same area as a vertical rectangle of height $rs^*L$ and width $\sigma+1/2$. (The ``$+1/2$'' arises from thinking of each lattice point as the center of a unit square.) Further, subtracting these two areas, each of which equals $rL \sqrt{(\sigma+1/2)(\tau+1/2)}$, gives a heuristic derivation of the order-$r$ correction term in the theorem. 

\begin{figure}
\includegraphics[scale=.4]{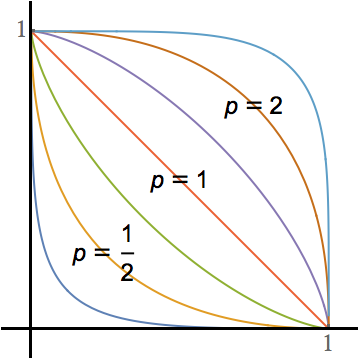}
\caption{\label{p-ellipsefig}The family of $p$-circles $x^p+y^p=1$, for $0<p<\infty$. \autoref{ex:circle_shift} and \autoref{ex:p_ellipse_shift} consider $p=2$ and $p=1/2$, respectively.}
\end{figure}
\begin{example}[Sufficient condition on shift parameters for the circle]\label{ex:circle_shift} When the curve $\Gamma$ is the portion of the unit circle in the first quadrant, one takes $L=1, f(x) = \sqrt{1-x^2}$, and $\alpha=\beta=1/\sqrt{2}$. Notice $f$ is smooth and concave, with monotonic second derivative. By symmetry it suffices to consider $\sigma \leq \tau$. When $\sigma \leq \tau\leq 0$, Parameter Assumption~\ref{pa:concave} says 
\[
\sqrt{1-\big(\frac{1+\sigma}{2+\sigma}\big)^2} < 2\sigma+2\tau + 1. 
\]
When $\sigma \leq 0\leq \tau$, equality in Parameter Assumption~\ref{pa:concave} would give a straight line. The resulting allowable region of $(\sigma,\tau)$-shift parameters for \autoref{th:S_limit_shift} is plotted on the left side of \autoref{fig:allowable_shift}.
\begin{figure}
  \includegraphics[width=.3\linewidth]{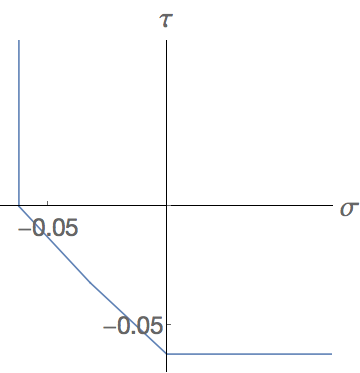}
 \hspace{1cm}
  \includegraphics[width=.3\linewidth]{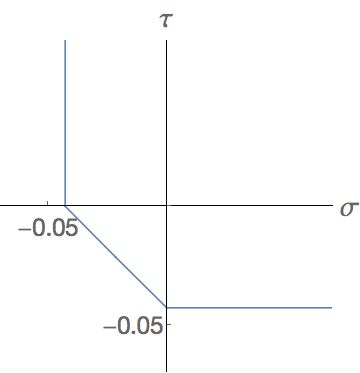}
\caption{\label{fig:allowable_shift}The allowable shift parameters $(\sigma, \tau)$ for \autoref{th:S_limit_shift} form the regions above the plotted curves, in the special cases where $\Gamma$ is a circle (figure on the left) and a $p$-circle with $p = 1/2$ (figure on the right). The intercepts are at approximately $-0.06$ (left figure) and $-0.04$ (right figure). The straight lines in the second and fourth quadrants are vertical and horizontal, respectively. The curves joining the intercepts are not quite straight lines. See \autoref{ex:circle_shift} and \autoref{ex:p_ellipse_shift}.}
\end{figure}
\end{example}
\begin{example}[Sufficient condition on shift parameters for $p$-circle with $p=1/2$]\label{ex:p_ellipse_shift}
Suppose $\Gamma$ is the part of the $1/2$-circle lying in the first quadrant, so that $L=1, f(x) = (1-x^{1/2})^2$, and take $\alpha=\beta=1/4$. Notice $f$ is smooth and convex, with monotonic second derivative $f''(x)=\frac{1}{2}x^{-3/2}$. The region of allowable shift parameters for \autoref{th:S_limit_shift} can be found numerically from Parameter Assumption~\ref{pa:convex}, as shown on the right side of \autoref{fig:allowable_shift}. 
\end{example}
Next we define weaker smoothness conditions. Let $(\alpha ,\beta)$ be a point on the curve $\Gamma$ with $\alpha,\beta>0$.
\begin{concavecondweak}\label{cond-weak-concave} Suppose $\Gamma$ is concave, and:
\begin{itemize}
 \item $f \in C^2(0,\alpha], f'<0, f''<0$, and a partition $0=\alpha_0<\alpha_1 < \dots < \alpha_l = \alpha$ exists such that $f''$ is monotonic on $(\alpha_{i-1},\alpha_i)$ for each $i=1,2,\dots,l$; 
 \item $g \in C^2(0,\beta], g'<0, g''<0$, and a partition $0=\beta_0<\beta_1 < \dots < \beta_m = \beta$ exists such that $g''$ is monotonic on $(\beta_{i-1},\beta_i)$ for each $i=1,2,\dots,m$; 
\item positive functions $\delta(r)$ and $\epsilon(r)$ exist such that
\begin{align}
\delta(r) = O(r^{-2a_1}) , \qquad & f''\big(\delta(r)\big)^{-1} = O(r^{1-4a_2}) , \label{eq:f_sup_shift}\\
\epsilon(r) = O(r^{-2b_1}) , \qquad & g''\big(\epsilon(r)\big)^{-1} = O(r^{1-4b_2}) , \label{eq:g_sup_shift}
\end{align}
as $r \to \infty$, for some numbers $a_1,a_2,b_1,b_2>0$;
\item and define $a_3=1/2, b_3=1/2$.
\end{itemize}
(The second condition in \autoref{eq:f_sup_shift} says that $f''(x)$ cannot be too small as $x \to 0$.)  
\end{concavecondweak}

\begin{convexcondweak}\label{cond-weak-convex} Suppose $\Gamma$ is convex, and:
\begin{itemize}
 \item $f \in C^2[\alpha,L), f'<0, f''>0$, and a partition $\alpha=\alpha_0<\alpha_1 < \dots < \alpha_l = L$ exists such that $f''$ is monotonic on $(\alpha_{i-1},\alpha_i)$ for each $i=1,2,\dots,l$;  
 \item $g \in C^2[\beta,L), g'<0, g''>0$, and a partition $\beta=\beta_0<\beta_1 < \dots < \beta_m = L$ exists such that $g''$ is monotonic on $(\beta_{i-1},\beta_i)$ for each $i=1,2,\dots,m$; 
\item positive functions $\delta(r)$ and $\epsilon(r)$ exist such that
\begin{align}
\delta(r) = O(r^{-2a_1}) , \qquad & f''\big(L-\delta(r)\big)^{-1} = O(r^{1-4a_2}) , \label{eq:f_sup_convex}\\
\epsilon(r) = O(r^{-2b_1}) , \qquad & g''\big(L-\epsilon(r)\big)^{-1} = O(r^{1-4b_2}) , \label{eq:g_sup_convex}
\end{align}
as $r \to \infty$, for some numbers $a_1,a_2,b_1,b_2>0$;
\item and suppose $f(x) = L+O(x^{2a_3})$ as $x \to 0$, and $g(y) = L+O(y^{2b_3})$ as $y \to 0$, for some numbers $a_3,b_3>0$.
\end{itemize}
(The last condition says $\Gamma$ cannot approach the axes too rapidly near the intercept points.)
\end{convexcondweak}

Concave Condition~\ref{cond:concave} implies Weaker Concave Condition~\ref{cond-weak-concave}, by choosing $\delta(r)=\epsilon(r)=r^{-1}$, $a_1 = b_1 = 1/2$ and $a_2 = b_2 = 1/4$, and noting that $f''(0) \neq 0, g''(0) \neq 0$. The same reasoning shows Convex Condition~\ref{cond:convex} implies Weaker Convex Condition~\ref{cond-weak-convex} with $a_3 = b_3 = 1/4$, since 
\[
g(L)=0, \ g'(L) \leq 0, \ g''(L)>0 \quad \Longrightarrow \quad g(L-y) \geq cy^2 \text{\ for small $y>0$}
\]
where $c>0$, and substituting $y=\sqrt{x/c}$ gives $L-f(x) \leq \sqrt{x/c}$ for small $x>0$, and similarly for $g$. 

Thus \autoref{th:S_limit_shift} follows immediately from the next result.
\begin{theorem}[Weaker conditions for asymptotic balance of optimal curve]\label{th:S_limit_shift_general}
If the curve $\Gamma$ and shift parameters $\sigma, \tau>-1/2 $ satisfy either Parameter Assumption~\ref{pa:concave} and Weaker Concave Condition~\ref{cond-weak-concave}, or Parameter Assumption~\ref{pa:convex} and Weaker Convex Condition~\ref{cond-weak-convex}, then the stretch factors maximizing $N(r,s)$ approach
\[
s^*=\sqrt{\frac{\tau + 1/2}{\sigma+1/2}}
\]
as $r \to \infty$, with
\[
S(r) \subset \big[ s^*-O(r^{-\se}), s^*+O(r^{-\se}) \big]
\] 
where 
\[
\se=\min \{ \tfrac{1}{6},a_1,a_2,a_3,b_1,b_2,b_3 \} .
\]
Further, the maximal lattice count has asymptotic formula
\begin{equation} \label{eq:maximalasymptotic}
\max_{s>0} N(r, s) = r^2 \area(\Gamma)-2rL \sqrt{(\sigma+1/2)(\tau+1/2)} + O(r^{1-2\se}).
\end{equation}
\end{theorem}
The proof in \autoref{section_proof} relies on lattice point counting propositions developed in \autoref{section_concave}, \autoref{section_convex} and \autoref{section_lower_bound}. 

\begin{example}[$p$-circles]\label{ex:p_shift}
Suppose $\Gamma$ is the part of the $p$-circle $|x|^p+|y|^p=1$ lying in the first quadrant. When $p>1$ the curve is concave, and satisfies Weaker Concave Condition~\ref{cond-weak-concave} by \cite[Example~5]{Lau_Liu17}. When $0<p<1$ it is convex and satisfies Weaker Convex Condition~\ref{cond-weak-convex} by \cite[Example~4.3]{AL17}, noting that $a_3=b_3=p/2$ since $f(x)=1+O(x^p)$ as $x\to 0$ and $g(y)=1+O(y^p)$ as $y \to 0$. 

Thus \autoref{th:S_limit_shift_general} applies to each $p$-circle, $p \neq 1$. The allowable shift parameters can be determined numerically from Parameter Assumption~\ref{pa:concave} or \ref{pa:convex}, as in \autoref{ex:circle_shift} and \autoref{ex:p_ellipse_shift}. 
\end{example}

Next we show there can be no ``universal'' allowable region of negative shifts for \autoref{th:S_limit_shift}. Specifically, for each choice of negative shifts $\sigma, \tau < 0$, no matter how close to zero, a curve exists whose optimal stretch parameters grow to infinity or shrink to $0$ as $r \to \infty$. That is, the optimal curve degenerates in the limit. 
\begin{theorem}[Negative shift: optimal concave curve can degenerate]\label{th:neg_shift} If $-1<\sigma<0, \tau > -1$, then a concave $C^2$-smooth curve $\Gamma$ exists, with intercepts at $L=1$, such that for each $\epsilon \in (0,1)$ one has
\begin{equation}\label{eq:s_unbounded}
S(r) \subset (0,r^{\epsilon-1}) \cup (r^{1-\epsilon},\infty) 
\end{equation}
for all large $r$.
\end{theorem}
The construction is given in \autoref{section:proof_neg}. The point of the theorem is that as soon as one of the shift parameters is negative, a concave curve exists for which the maximizing stretch parameters approach either $0$ or $\infty$ as $r \to \infty$. 

For convex curves, we do not know an analogue of \autoref{th:neg_shift}: does a universal allowable region of $(\sigma,\tau)$ parameters exist in which \autoref{th:S_limit_shift} holds for all $C^2$-smooth convex decreasing curves? 

The ``bad'' curve in \autoref{th:neg_shift} can even be a quarter circle: 
\begin{proposition}[Negative shift: the optimal ellipse can degenerate]\label{prop:neg_ellipse}
If the curve $\Gamma$ is the quarter unit circle, and $\sigma , \tau>-1$ with either $\sigma \leq -2/5$ or $\tau \leq -2/5$, then for each $\epsilon \in (0,1)$ one has
\[
S(r) \subset (0,r^{\epsilon-1}) \cup (r^{1-\epsilon},\infty) \qquad \text{for all large $r$.}
\]
\end{proposition}
The proof is in \autoref{section:proof_neg}. And in \autoref{section_spectral} we apply this result to Laplacian eigenvalue minimization on rectangles.

\section{\bf Concave curves --- counting function estimates} \label{section_concave}
In order to prove \autoref{th:S_limit_shift_general} we need to estimate the counting function. The curve $\Gamma$ is taken  to be concave decreasing in the first quadrant, throughout this section. Denote the horizontal and vertical intercepts by $x=L$ and $y=M$ respectively, where $L$ and $M$ are positive but not necessarily equal. Allowing unequal intercepts is helpful for some of the results below. 

We start with a preliminary $r$-dependent bound on the maximizing set $S(r)$. The proof of this bound also makes clear why $N(r,s)$ attains its maximum as a function of $s$, for each fixed $r$, so that the set $S(r)$ is well defined. 

\begin{lemma}[Linear-in-$r$ bound on optimal stretch factors for concave curves]\label{lemma:bound_shift}If $\sigma, \tau > -1$ then
\[
S(r)\subset \big[(1+\tau)/rM, rL/(1+\sigma)\big] \qquad \text{whenever $r \geq (2+\sigma+\tau)/\sqrt{LM}$.}
\]
\end{lemma}
\begin{proof}
The curve $r\Gamma(s)$ with the particular choice $s=\sqrt{L/M}$ has horizontal and vertical intercepts equal to
$r\sqrt{LM}$. That intercept value is $\geq (2+\sigma+\tau)$, by assumption on $r$ in this lemma. Hence by
concavity, $r\Gamma(s)$ encloses the  point $(1+\sigma,1+\tau)$ and so $N(r,s)>0$ for this particular value of $s$, which means the maximum of $s \mapsto N(r,s)$ is greater than $0$. 

When $s> rL/(1+\sigma)$, the $x$-intercept of $r\Gamma(s)$ is less than $1+\sigma$ and so no shifted lattice points are enclosed, meaning $N(r,s)=0$. Thus the maximum is not attained for such $s$-values. Arguing similarly with the $y$-intercept shows the maximum is also not attained when $s<(1+\tau)/rM$. The lemma follows. 
\end{proof}

The last lemma required only that $\Gamma$ be concave decreasing. Smoothness was not needed. Smoothness is not used in the next proposition either, which gives an upper bound on the counting function and so extends a result from the unshifted case \cite[Proposition~10]{Lau_Liu17}. 
\begin{proposition}[Two-term upper bound on counting function for concave curves]\label{prop:ineq_shift}
Let $\sigma, \tau > -1$. The number $N(r, s)$ of shifted lattice points lying inside $r\Gamma(s)$ satisfies
\begin{equation}\label{eq:two_term_ineq}
N(r,s) \leq r^2 \area(\Gamma)-C_1rs+\sigma^- \tau^-
\end{equation}
for all $r \geq (1-\sigma^-)s/L$ and $s \geq 1$, where 
\begin{equation}
C_1 = C_1(\Gamma, \sigma, \tau) = \frac{1}{2}\Big(M-f(\frac{1-\sigma^-}{2-\sigma^-}L)\Big)-\sigma^-M-\tau^-L.
\end{equation} 
\end{proposition}
The constant $C_1$ might or might not be positive. Parameter Assumption~\ref{pa:concave} consists of the assumption $C_1>0$ along with the corresponding inequality for $g$, in the situation where $L=M$.
\begin{proof}
First suppose $\sigma\leq 0, \tau \leq 0$. Write $N$ for the number of shifted lattice points under $\Gamma$, and suppose $L\geq 1+\sigma$ so that $\lfloor L-\sigma \rfloor \geq 1$. Extend the curve $\Gamma$ horizontally from $(0, M)$ to $(\sigma, M)$, so that $f(\sigma)=M$. Construct triangles with vertices at $\big(i-1+ \sigma, f(i-1+\sigma)\big), \big(i+\sigma, f(i+\sigma)\big), \big(i-1+\sigma, f(i+\sigma)\big)$ for $i = 1,\dots ,\lfloor L -\sigma \rfloor$, as illustrated in \autoref{fig:area-comp}. The rightmost vertex of the final triangle has horizontal coordinate $\lfloor L-\sigma \rfloor + \sigma$, which is less than or equal to $L$. These triangles lie above the unit squares with upper right vertices at shifted lattice points, and lie below the curve $\Gamma$ due to concavity. 
\begin{figure}[t]
 \includegraphics[scale=0.4]{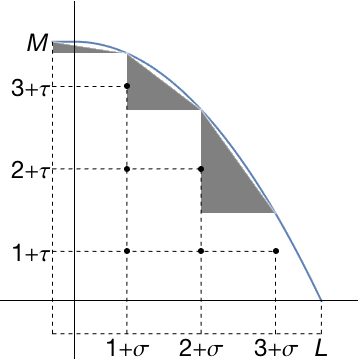}
 \caption{\label{fig:area-comp}Concave curve enclosing lattice points shifted in the negative direction. The square areas represent the lattice point count, while the triangles estimate the discrepancy between that count and the area under the curve, as needed for \autoref{prop:ineq_shift}}. 
\end{figure}
Hence 
\begin{equation}\label{eq:total_area_shift}
N + \area(\text{triangles}) \leq \area(\Gamma)-\sigma(M-\tau)-\tau(L-\sigma) -\sigma\tau,
\end{equation}
where the correction terms on the right side of the inequality represent the areas of the rectangular regions outside the first quadrant. 

Letting $k = \lfloor L-\sigma\rfloor \geq 1$, we compute
\begin{align}
\area(\text{triangles})&=\sum_{i=1}^{k} \frac{1}{2}\big(f(i-1+\sigma) - f(i+\sigma)\big) \nonumber\\
&=\frac{1}{2}\big(M-f(k+\sigma)\big) \notag 
\\
&\geq \frac{1}{2}\big(M- f(\frac{1+\sigma}{2+\sigma}L)\big) \label{eq:area_shift}
\end{align}
because $f$ is decreasing and $k+\sigma \leq L< k+1+\sigma$ implies 
\[
k+\sigma > \frac{k+\sigma}{k+1+\sigma} L \geq \frac{1+\sigma}{2+\sigma} L .
\]
Combining \autoref{eq:total_area_shift} and \autoref{eq:area_shift} proves 
\begin{equation}\label{eq:n_bound}
N \leq \area(\Gamma) - \sigma M -\tau L -\frac{1}{2}\Big(M- f(\frac{1+\sigma}{2+\sigma}L)\Big) +\sigma\tau .
\end{equation}
Now we replace $\Gamma$ with the curve $r\Gamma(s)$, meaning we replace $N, L, M, f(x)$ with $N(r,s), rs^{-1}L, rsM, rsf(sx/r)$ respectively, thereby obtaining the desired estimate \autoref{eq:two_term_ineq} (noting that $L/s \leq Ls$ since $s \geq 1$). The restriction $L\geq 1+\sigma$ becomes $r \geq (1+\sigma)s/L$ under the rescaling, and so we have proved the proposition in the case $\sigma \leq 0, \tau \leq 0$. 

When $\sigma> 0, \tau> 0$, the number of shifted lattice points inside $r\Gamma(s)$ is less than or equal to the number when there is no shift ($\sigma=\tau=0$), simply because the curve is decreasing. Thus this case of the proposition follows from the ``$\sigma,\tau \leq 0$'' case above. 

When $\sigma > 0, \tau \leq 0$, the number of shifted lattice points inside $r\Gamma(s)$ is less than or equal to the number for $\sigma = 0$ with the same $\tau$ value, and so this case of the proposition follows also from the ``$\sigma,\tau \leq 0$'' case above. A similar argument holds when $\sigma \leq 0, \tau > 0$.  
\end{proof}
\begin{corollary}[Improved two-term upper bound on counting function for concave curves]\label{cor:ineq_shift}
Let $\sigma, \tau > -1$. If $s$ is bounded above and bounded below away from $0$, as $r \to \infty$, then the number $N(r, s)$ of shifted lattice points lying inside $r\Gamma(s)$ satisfies
\begin{equation}\label{eq:boundedimprove}
N(r,s) \leq r^2\area(\Gamma)-r\big(s^{-1} \tau L + s(\sigma+1/2)M \big)+ o(r) .
\end{equation}
\end{corollary}
\begin{proof}
Take $c>1$ and suppose $c^{-1}<s<c$ throughout the rest of the proof. 

Suppose $\sigma,\tau \leq 0$. Let $K \geq 1$. Repeat the proof of \autoref{prop:ineq_shift} except with the initial supposition $L \geq 1+\sigma$ replaced by $L \geq K + \sigma$, and do not assume $s \geq 1$. One finds 
\[
N(r,s) \leq r^2 \area(\Gamma)-D_K rs-\tau L r s^{-1} + \sigma \tau
\]
for all $r \geq (K+\sigma)s/L$, where 
\[
D_K = D_K(\Gamma, \sigma) = \frac{1}{2}\Big(M-f(\frac{K+\sigma}{K+1+\sigma}L)\Big)+\sigma M.
\]
We deduce
\begin{align*}
& \limsup_{r \to \infty} \sup_{s<c} \frac{1}{r}\Big( N(r,s) - r^2\area(\Gamma) + r\big(s(\sigma+1/2)M +s^{-1} \tau L \big) \Big) \\
& \leq \frac{c}{2} f(\frac{K+\sigma}{K+1+\sigma}L) .
\end{align*}
The last expression can be made arbitrarily small by choosing $K$ sufficiently large (recall $f(L)=0$), and so the left side is $\leq 0$. That proves the corollary when $\sigma,\tau \leq 0$. 

Suppose $\sigma >0, \tau\leq 0$. We will relate this case to the previous one. To emphasize the dependence of the counting function on the shift parameters, write $N_{\sigma,\tau}(r,s)$ for the counting function that was previously written $N(r,s)$. Adding  columns of shifted lattice points at $x=\sigma-\lceil \sigma \rceil + 1, \dots, \sigma-1, \sigma$ gives the counting function $N_{\widetilde{\sigma},\tau}(r,s)$ where $\widetilde{\sigma}=\sigma-\lceil \sigma \rceil \in (-1,0]$. This counting function is related to the original one by
\begin{align*}
N_{\widetilde{\sigma},\tau}(r,s)&=N_{\sigma,\tau}(r,s) + \sum_{i=0}^{\lceil \sigma \rceil - 1} \lfloor rsf\big(s(\sigma-i)/r\big)-\tau \rfloor,\\
&=N_{\sigma,\tau}(r,s) + \lceil \sigma \rceil rsM+o(r),
\end{align*}
as $r\to \infty$, since $s$ is bounded above and $f$ is continuous with $f(0)=M$. Since $\widetilde{\sigma},\tau\leq 0$, we may apply \autoref{eq:boundedimprove} with $\sigma$ replaced by $\widetilde{\sigma}$ to obtain
\[
N_{\widetilde{\sigma},\tau}(r,s) \leq r^2\area(\Gamma)-r\big(s^{-1} \tau L + s(\sigma-\lceil \sigma \rceil +1/2)M \big)+ o(r) \quad \text{as $r\to \infty$}.
\]
Combining the above two formulas, we prove the corollary for $\sigma>0, \tau \leq 0$.
 
When $\sigma \leq 0, \tau > 0$, simply add the appropriate rows instead of columns and argue like above using $\lceil \tau \rceil$ instead of $\lceil \sigma \rceil$, and using the boundedness of $s^{-1}$. Similarly, one can treat the case $\sigma>0, \tau > 0$. 
\end{proof}

The next proposition gives an asymptotic approximation to $N(r,s)$, assuming the curve is concave decreasing and has suitably monotonic second derivative.  
\begin{proposition}[Two-term counting estimate for concave curves]\label{th:asy_pos}
Let $\sigma, \tau> -1$ and $0 \leq q < 1$. If Weaker Concave Condition~\ref{cond-weak-concave} holds and $s+s^{-1} = O(r^q)$ then 
\begin{align}\label{eq:asy_pos}
&N(r,s)=r^2\area (\Gamma)-r\big( s^{-1}(\tau + 1/2)L+s(\sigma+1/2)M\big)+ O(r^Q) 
\end{align}
as $r \to \infty$, where
\[
Q = \max\{\tfrac{2}{3}, \tfrac{1}{2}+\tfrac{3}{2}q , 1-2a_1+q,1-2a_2+\tfrac{3}{2}q, 1-2b_1+q, 1-2b_2+\tfrac{3}{2}q \}.
\]
Special cases: (i) If $q=0$ then $Q = 1-2e$ where $e = \min\{ \tfrac{1}{6}, a_1, a_2,b_1, b_2 \}$. \\
(ii) If Concave Condition~\ref{cond:concave} holds then $Q=\max\{\tfrac{2}{3}, \tfrac{1}{2}+\tfrac{3}{2}q \}$. 
\end{proposition}
The numbers $a_1,a_2,b_1,b_2$ come from Weaker Concave Condition~\ref{cond-weak-concave}. That Condition also involves a point $(\alpha,\beta) \in \Gamma$ with $\alpha,\beta>0$, which we use in the following proof. 
\begin{proof}
The idea is to translate and truncate the curve $r\Gamma(s)$ as in \autoref{fig:neg_shift}, in order to reduce to an unshifted lattice problem.  Then we invoke known results from our earlier paper \cite{Lau_Liu17} (which builds on work of Kr\"{a}tzel \cite{kratzel00,kratzel04} and a theorem of van der Corput). 

\smallskip Step 1 --- Translating and truncating.
Notice $rs \to \infty$ and $rs^{-1} \to \infty$ as $r \to \infty$, since $s=O(r^q)$ and $s^{-1}=O(r^q)$ with $q<1$. Thus by taking $r$ large enough, we insure
\[
rs^{-1} g \Big( s^{-1} \frac{1+\tau}{r} \Big) > rs^{-1} \alpha > 1+\sigma , \qquad
rs f \Big( s \frac{1+\sigma}{r} \Big) > rs \beta > 1+\tau .
\]
For all large $r$ one also has $\delta(r)<\alpha$ and $\epsilon(r)<\beta$, by Weaker Concave Condition~\ref{cond-weak-concave}. 

Given a large $r$ satisfying the above conditions, and a corresponding $s>0$, we let
\[
\widetilde{\alpha} = rs^{-1} \alpha - (1+\sigma) , \qquad
\widetilde{\beta} = rs \beta - (1+\tau) ,
\]
and
\[
\widetilde{L} = rs^{-1} g \Big( s^{-1} \frac{1+\tau}{r} \Big) - (1+\sigma) , \qquad
\widetilde{M} = rs f \Big( s \frac{1+\sigma}{r} \Big) - (1+\tau) ,
\]
so that 
\[
0 < \widetilde{\alpha} < \widetilde{L}, \qquad 0 < \widetilde{\beta} < \widetilde{M} .
\]
\begin{figure}
 \includegraphics[scale=0.4]{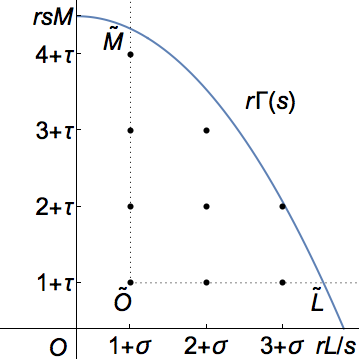}
 \caption{\label{fig:neg_shift}Curve $r\Gamma(s)$ enclosing positive-integer lattice points shifted by $(\sigma,\tau)=(-0.3,-0.4)$. The new origin is $\sO$, and $\sL$ and $\sM$ are the new $x$- and $y$-intercepts, as defined in the proof of \autoref{th:asy_pos}.}
\end{figure}

Consider the point $\widetilde{O}=(1+\sigma,1+\tau)$ in the first quadrant. Regard this point as the new origin, and let $\widetilde{\Gamma}$ be the portion of $r\Gamma(s)$ lying in the new first quadrant --- see \autoref{fig:neg_shift}. That is, $\widetilde{\Gamma}$ is the graph of 
\[
\widetilde{f}(x) = rs f \Big( s \frac{x+1+\sigma}{r} \Big) - (1+\tau) , \qquad 0 \leq x \leq \widetilde{L} ,
\]
and also of its inverse function 
\[
\widetilde{g}(y) = rs^{-1} g \Big( s^{-1} \frac{y+1+\tau}{r} \Big) - (1+\sigma) , \qquad 0 \leq y \leq \widetilde{M} .
\]
Notice $(\widetilde{\alpha},\widetilde{\beta}) \in \widetilde{\Gamma}$, since $f(\widetilde{\alpha})=\widetilde{\beta}$. Write $\widetilde{N}$ for the number of positive-integer lattice points under the curve $\widetilde{\Gamma}$. That is, 
\[
\widetilde{N} = \# \{ (j,k) \in \N \times \N : k \leq \widetilde{f}(j) \} .
\]
This $\widetilde{N}$ does not count the lattice points in the first column or row, which arise from $j=0$ or $k=0$.

Weaker Concave Condition~\ref{cond-weak-concave} guarantees that $\widetilde{f}$ is $C^2$-smooth on the interval $[0,\widetilde{\alpha}]$, with $\widetilde{f}^\prime < 0$ and $\widetilde{f}^{\prime \prime}<0$ there, and similarly $\widetilde{g}$ is $C^2$-smooth on $[0,\widetilde{\beta}]$ with $\widetilde{g}^\prime < 0$ and $\widetilde{g}^{\prime \prime}<0$ there. 

Next, we partition the interval $[0,\widetilde{\alpha}]$ as $0=\widetilde{\alpha}_0 < \widetilde{\alpha}_1 < \dots < \widetilde{\alpha}_{\widetilde{l}} = \widetilde{\alpha}$ where the interior partition points are chosen to be the elements of
\[
\{ rs^{-1} \alpha_i - (1+\sigma) : i=1,\dots,l-1 \} 
\]
that happen to lie between $0$ and $\widetilde{\alpha}$. Observe $\widetilde{f}^{\prime \prime}$ is monotonic on each subinterval of the partition, by Weaker Concave Condition~\ref{cond-weak-concave}. Similarly, $\widetilde{g}^{\prime \prime}$ is monotonic on each subinterval of the corresponding partition $0=\widetilde{\beta}_0 < \widetilde{\beta}_1 < \dots < \widetilde{\beta}_{\widetilde{m}} = \widetilde{\beta}$ of the interval $[0,\widetilde{\beta}]$.

Let 
\[
\widetilde{\delta} = \big[ rs^{-1} \delta(r) - (1+\sigma) \big]^+ , \qquad
\widetilde{\epsilon} = \big[ rs \epsilon(r) - (1+\tau) \big]^+ ,
\]
so that $0 \leq \widetilde{\delta}<\widetilde{\alpha}$ and $0 \leq \widetilde{\epsilon}<\widetilde{\beta}$. 

To relate some of these old and new quantities, we denote antiderivatives of $f, g$ by
\begin{equation}\label{eq:F_def}
F(x) = \int_0^x f(t) \ud t,  \qquad G(y) = \int_0^y g(t) \ud t ,
\end{equation}
and observe that
\begin{align*}
\area(\sGamma) &=r^2\area(\Gamma) - r^2 \big(F((1+\sigma) s/r)+ G((1+\tau) s^{-1}/r)\big)+ (1+\sigma)(1+ \tau),\\
\shiftf '(x)&= s^2f'\big(s\frac{x+1+\sigma}{r}\big),\qquad
\shiftf ''(x)= \frac{s^3}{r}f''\big(s\frac{x+1+\sigma}{r}\big),\\
\int_0^{\sa} | \shiftf ''(x)|^{1/3} \ud x &= r^{2/3} \int_{(1+\sigma)s/r}^{\alpha} | f ''(x)|^{1/3} \ud x \leq r^{2/3} \int_0^{\alpha} | f
''(x)|^{1/3} \ud x,\\
\sum_{i=1}^{\widetilde{l}}\frac{1}{|\shiftf''(\sa_i)|^{1/2}} & \leq\sum_{i=1}^l\frac{r^{1/2}s^{-3/2}}{|f''(\alpha_i)|^{1/2}},
\end{align*}
and similarly for $\sg$ except with $s$ replaced by $s^{-1}$.

\smallskip Step 2 --- Estimating the counting function.
Applying part (a) of \cite[Proposition~12]{Lau_Liu17} to the curve $\sGamma$ and using the preceding relationships, we get
\begin{align}
&\big|\sN-r^2\area (\Gamma)+ r^2 \big(F((1+\sigma) s/r)+ G((1+\tau) s^{-1}/r)\big) \nonumber\\
&\hspace{3cm}+ \frac{r}{2}\big(s f((1+\sigma)s/r) + s^{-1}g((1+\tau)s^{-1}/r) \big) \big|\nonumber\\
&\leq 6r^{2/3}\Big(\int_0^\alpha |f''(x)|^{1/3} \ud x+\int_0^\beta |g''(y)|^{1/3}\ud y\Big)  +175r^{1/2}\big(\frac{s^{-3/2}}{|f''(\delta(r))|^{1/2}}+\frac{s^{3/2}}{|g''(\epsilon(r))|^{1/2}}\big)\nonumber\\
& \hspace{.5cm}+525r^{1/2}\big(\sum_{i=1}^l\frac{s^{-3/2}}{|f''(\alpha_i)|^{1/2}}+\sum_{j=1}^m\frac{s^{3/2}}{|g''(\beta_j)|^{1/2}}\big) +\frac{1}{4}(\sum_{i=1}^l s^2|f'(\alpha_i)|+\sum_{j=1}^m s^{-2}|g'(\beta_j)|)\nonumber\\
&\hspace{.5cm}+\frac{r}{2}(s^{-1}\delta(r)+s\epsilon(r))+l+m
+ \frac{1}{2}(1+\sigma)+\frac{1}{2}(1+\tau)+(1+\sigma)(1+\tau)+1 , \label{eq:two-term-asym-exact-concave}
\end{align}
where we dealt with the term involving $|\shiftf''(\widetilde{\delta})|^{-1/2}$ in \cite[Proposition~12]{Lau_Liu17} as follows. One has $\shiftf''(\widetilde{\delta})=r^{-1}s^3 f''(z)$ where $z = r^{-1}s(\widetilde{\delta}+1+\sigma) \geq \delta(r)$, and so by monotonicity of $f''$ on each subinterval of the partition (as assumed in Weaker Concave Condition~\ref{cond-weak-concave}) one concludes
\[
|\shiftf''(\widetilde{\delta})|\geq r^{-1}s^3 \min \{ |f''\big(\delta(r)\big)|,  |f''\big(\alpha_1\big)| , \dots,  |f''\big(\alpha_l \big)| \}.
\]
Thus the term involving $|\shiftf''(\widetilde{\delta})|^{-1/2}$ can be estimated by the sum of terms involving $|f''\big(\delta(r)\big)|^{-1/2}$ and $|f''\big(\alpha_i\big)|^{-1/2}$. 

The right side of \autoref{eq:two-term-asym-exact-concave} already has the desired order $O(r^Q)$, by direct estimation and using that $s+s^{-1}=O(r^q)$ and $2q < \tfrac{1}{2}+\tfrac{3}{2}q$ since $q<1$. 

\smallskip Step 3 --- Understanding the left side of inequality \autoref{eq:two-term-asym-exact-concave}. 
It remains to deal with the terms on the left of \autoref{eq:two-term-asym-exact-concave}. Clearly $N(r,s)$ and $\sN$ count the same lattice points, except that $N(r,s)$ also counts the points in the first row and column. That is, 
\begin{align*}
\sN 
& = N(r,s) - \lfloor rsf\big((1+\sigma)s/r\big)-\tau \rfloor - \lfloor rs^{-1} g\big((1+\tau)s^{-1}/r\big)-\sigma\rfloor +1 \\
& = N(r,s) -  rsf\big((1+\sigma)s/r\big) -\tau - rs^{-1} g\big((1+\tau)s^{-1}/r\big)-\sigma + \rho(r,s) 
\end{align*}
for some number $\rho(r,s) \in [1,3]$. Substitute this formula into the left side of \autoref{eq:two-term-asym-exact-concave}. Substitute also the following expressions, which are obtained from \autoref{lemma:f-estimate}: 
\begin{align*}
rsf((1+\sigma)s /r) &=rsM +O(s^2),\\
r^2F((1+\sigma) s/r) &=rs (1+\sigma) M +O(s^2),
\end{align*}
and similarly for $g$ and $G$.
The proposition now follows straightforwardly, since $O(s^2)=O(r^{2q})$. 
\end{proof}

\begin{lemma}\label{lemma:f-estimate}
If $f$ is decreasing and concave on $[0,L]$ then 
\[
f(x)=f(0) +O(x), \qquad F(x)=f(0)x + O(x^2), \qquad \text{as $x \to 0$,}
\]
where $F(x) = \int_0^x f(t) \ud t$ is the antiderivative of $f(x)$. 
\end{lemma}
\begin{proof}
The difference quotient $(f(x)-f(0))/x$ is a decreasing function of $x$ since $f$ is concave, and it is less than or equal to $0$ since $f$ is decreasing. Hence the difference quotient is bounded, and so $f(x)=f(0) +O(x)$. Integrating completes the proof. 
\end{proof}

\section{\bf Convex curves --- counting function estimates} \label{section_convex}
Assume the curve $\Gamma$ is convex decreasing, throughout this section. We will prove estimates for convex curves analogous to the work in \autoref{section_concave} for concave curves. 

\autoref{lemma:r_bound_negative} below is an improved  $r$-dependent bound on the optimal stretch factors, generalizing Ariturk and Laugesen's lemma from the unshifted situation \cite[Lemma~7.2]{AL17}. By ``improved'' we refer to the upper and lower bounds: for instance, when $\sigma =0$ the upper bound in \autoref{lemma:r_bound_negative} improves on the bound in \autoref{lemma:bound_shift} by a factor of $2$. This tighter bound on the optimal stretch factor gives us more flexibility when deriving the two-term counting estimate in \autoref{prop:ineq_convex_shift}. 

In the next lemma we assume for simplicity that the $x$- and $y$-intercepts are both $L$, so that we need not change the definitions of $\mu_f(\sigma)$ and $\mu_g(\tau)$ in \autoref{section:results}.
\begin{lemma}[Improved linear-in-$r$ bound on optimal stretch factors for convex curves]\label{lemma:r_bound_negative}
If $\sigma, \tau > -1$ with $\mu_f(\sigma)> 0 $ and $\mu_g(\tau) > 0$, then
\[
S(r)\subset \Big[ \frac{2+\tau}{rL} , \frac{rL}{2+\sigma} \Big]
\]
whenever 
\begin{equation} \label{eq:improvedcondition}
r \geq \max\Big((2+\sigma)\sqrt{2(1+\tau)/L\mu_f(\sigma)}, (2+\tau)\sqrt{2(1+\sigma)/L\mu_g(\tau)}\Big) .
\end{equation}
\end{lemma}
\begin{proof}\ 

Claim~1: $N(r,s)=0$ if $s \in \big( 0,(1+\tau)/rL \big]$ or $s \in \big[ rL/(1+\sigma),\infty \big)$. Indeed, the curve $r\Gamma(s)$ has $x$- and $y$-intercepts at $rL/s$ and $rsL$, respectively, and so if $rL/s \leq 1+\sigma$ or $rsL \leq 1+\tau$ then the point $(1+\sigma,1+\tau)$ is not enclosed by the curve and so the lattice count $N(r,s)$ is zero. 

Claim~2: if \autoref{eq:improvedcondition} holds and $s\in \big( rL/(2+\sigma),rL/(1+\sigma) \big)$ then 
\[
N(r,s) < N\Big(r,\frac{1+\sigma}{2+\sigma} s\Big) .
\]
To prove this claim, notice the $x$-intercept satisfies 
\[
1+\sigma < \frac{rL}{s} < 2+\sigma ,
\]
and so only the first column of shifted lattice points (the points with $x$-coordinate at $1+\sigma$) can contribute to the count inside $r\Gamma(s)$. Hence $N(r,s)= \lfloor rsf((1+\sigma)s/r) - \tau\rfloor$. Meanwhile, if we count shifted lattice points in the first two columns (where $x=1+\sigma$ and $x=2+\sigma$) we find 
\begin{align}
& N\Big(r,\frac{1+\sigma}{2+\sigma} s\Big) \\
&\geq \big\lfloor rs\frac{1+\sigma}{2+\sigma}f\Big( \frac{(1+\sigma)^2s}{(2+\sigma )r}\Big)- \tau \big\rfloor + \big\lfloor rs\frac{1+\sigma}{2+\sigma}f\Big( \frac{(1+\sigma)s}{r}\Big) - \tau \big\rfloor \nonumber \\
&> rs\frac{1+\sigma}{2+\sigma}f\Big( \frac{(1+\sigma)^2s}{(2+\sigma )r}\Big) + rs\frac{1+\sigma}{2+\sigma}f\Big( \frac{(1+\sigma)s}{r}\Big) -2\tau -2 \nonumber \\
&= rsf\Big(\frac{(1+\sigma)s}{r}\Big) + \frac{rs}{2+\sigma}\Big( (1+\sigma)f\Big( \frac{(1+\sigma)^2s}{(2+\sigma )r}\Big) -f\Big( \frac{(1+\sigma)s}{r}\Big) \Big) -2(1+\tau)  \nonumber\\
&\geq rsf\Big(\frac{(1+\sigma)s}{r}\Big) + \frac{rs}{2+\sigma}\mu_f(\sigma) -2(1+\tau) \notag \\
&>  rsf\Big(\frac{(1+\sigma)s}{r}\Big) \geq N(r,s) , \notag
\end{align}
where to get the final line we use that $\frac{rs}{2+\sigma}\mu_f(\sigma)>2(1+\tau)$, which follows from $s>rL/(2+\sigma)$ and the lower bound on $r$ in \autoref{eq:improvedcondition}. The proof of Claim~2 is complete. 

Claim~3: if \autoref{eq:improvedcondition} holds and $s \in \big( (1+\tau)/rL , (2+\tau)/rL \big)$ then 
\[
N(r,s) < N\Big(r,\frac{2+\tau}{1+\tau} s \Big) .
\]
The proof is analogous to Claim~2, except counting in rows instead of columns.

Claim~4: if \autoref{eq:improvedcondition} holds then the maximizing $s$-values for $N(r,s)$ lie in the interval $\big[ (2+\tau)/rL, rL/(2+\sigma) \big]$. To see this, note that $N(r,s^\prime)>0$ for some $s^\prime>0$, by the strict inequality in Claim~2, and so the maximum does not occur in the intervals considered in Claim~1. The maximum does not occur in the interval considered in Claim~2, as that claim itself shows, and similarly for Claim~3. Thus the maximum must occur in the remaining interval, which proves Claim~4 and thus finishes the proof of the lemma. 
\end{proof}

The next bound generalizes work of Ariturk and Laugesen \cite[Proposition~5.1]{AL17} from the unshifted situation ($\sigma=\tau=0$) to the shifted case.
\begin{proposition}[Two-term upper bound on counting function for convex curves]\label{prop:ineq_convex_shift}
Let $\sigma,\tau> -1$. The number $N(r, s)$ of shifted lattice points lying inside $r\Gamma(s)$ satisfies
\begin{equation}\label{eq:two_term_ineq_convex}
N(r,s) \leq r^2 \area(\Gamma)-C_2rs + \sigma^- \tau^-
\end{equation}
for all $r \geq (2-\sigma^-)s/L$ and $s\geq 1$, where 
\[
C_2 = C_2(\Gamma, \sigma, \tau) = \frac{1}{2}(1-\sigma^-)f(\frac{1-\sigma^-}{2-\sigma^-}L)-\sigma^-M-\tau^-L .
\]
\end{proposition}
The constant $C_2$ need not be positive. That is why hypothesis \autoref{eq:convex_sigma} in Parameter Assumption~\ref{pa:convex} includes (for $L=M$) the assertion that $C_2>0$.  
\begin{figure}[t]
\includegraphics[scale=.4]{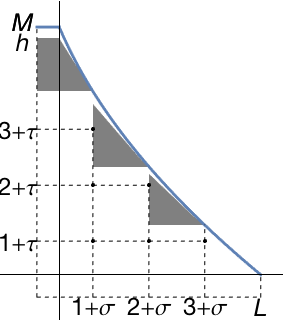}
\caption{\label{fig:triangleconvex}Convex curve enclosing lattice points shifted in the negative direction. The square areas represent the lattice point count, while the triangles and trapezoid estimate the discrepancy between that count and the area under the curve in \autoref{prop:ineq_convex_shift} }
\end{figure}
\begin{proof}
First consider $\sigma\leq0, \tau \leq0$. Write $N$ for the number of shifted lattice points under $\Gamma$. Suppose $L\geq 2+\sigma$. Extend the curve horizontally from $(0, M)$ to $(\sigma, M)$, so that $f(\sigma)=M$. Construct a trapezoid (see \autoref{fig:triangleconvex}) with vertices at $\big( \sigma,f(1+\sigma) \big)$, $\big( 1+\sigma, f(1+\sigma) \big)$, $(0,h)$, $(\sigma,h)$ where $h=f(1+\sigma)-(1+\sigma) f'(1+\sigma)$. Also construct triangles with vertices $\big(i-1+\sigma, f(i+\sigma)\big)$, $\big(i+\sigma, f(i+\sigma)\big)$, $\big(i-1+ \sigma, f(i+\sigma)-f'(i+\sigma)\big)$, where $i = 2,\dots ,\lfloor L -\sigma \rfloor$. These triangles lie above the squares with upper right vertices at the shifted lattice points, and like  below the curve by convexity, as \autoref{fig:triangleconvex} illustrates. Hence 
\begin{equation}\label{eq:total_area_convex}
N + \area(\text{trapezoid and triangles}) \leq \area(\Gamma)-\sigma(M-\tau)-\tau(L-\sigma) -\sigma\tau
\end{equation}

Let $k = \lfloor L-\sigma\rfloor \geq 2$, so that $k+\sigma \leq L < k+\sigma+1$. Then
\begin{align*}
\area(\text{trapezoid})
&= \frac{1}{2}(\text{base}+\text{top}) \cdot (\text{height}) \\
& = - \frac{1}{2} (1-\sigma) \cdot (1+\sigma) f'(1+\sigma) \\
& \geq \frac{1}{2} (1+\sigma) \big( f(1+\sigma) - f(2+\sigma) \big)
\end{align*}
by convexity, and using that $1-\sigma \geq 1$. Further, convexity implies
\begin{align}
\area(\text{triangles})&= - \frac{1}{2} \sum_{i=2}^{k}  f'(i+\sigma) \nonumber\\
& \geq \frac{1}{2}  \sum_{i=2}^{k-1} \big(f(i+\sigma)-f(i+1+\sigma) \big) + \frac{1}{2}\big(f(k+\sigma)-f(L) \big)\nonumber\\
&=\frac{1}{2}f(2+\sigma) .
\end{align}
Hence
\begin{align}
& \area(\text{trapezoid}) + \area(\text{triangles}) \\
& \geq \frac{1}{2}(1+\sigma)f(1+\sigma) - \frac{1}{2} \sigma f(2+\sigma) \notag \\
& \geq \frac{1}{2}(1+\sigma)f(\frac{1+\sigma}{2+\sigma}L) - \frac{1}{2} \sigma f(\frac{2+\sigma}{2+\sigma}L) \label{eq:triangle_area_estimate_convex}
\end{align}
since $f$ is decreasing and $L/(2+\sigma) \geq 1$. Combining \autoref{eq:total_area_convex} and \autoref{eq:triangle_area_estimate_convex} and using $f(L)=0$ proves 
\begin{equation*}\label{eq:n_bound_convex}
N \leq \area(\Gamma) - \sigma M -\tau L -\frac{1}{2}(1+\sigma)f\big(\frac{1+\sigma}{2+\sigma}L\big) +\sigma\tau .
\end{equation*}

Now replace $\Gamma$ with the curve $r\Gamma(s)$, meaning replace $N, L, M, f(x)$ with $N(r,s)$, $rs^{-1}L$, $rsM$, $rsf(sx/r)$ respectively. Using $s \geq 1$, we know $L/s \leq Ls$; the assumption $L \geq 2+\sigma$ becomes $r \geq (2+\sigma)s/L$. Thus we obtain \autoref{eq:two_term_ineq_convex} in the case $\sigma \leq 0, \tau \leq 0$. 

One may now deduce the remaining cases as was done in the proof of \autoref{prop:ineq_shift}. 
\end{proof}
\begin{corollary}[Improved two-term upper bound on counting function for convex curves]\label{cor:ineq_convex_shift}
Let $\sigma, \tau > -1$. If $s$ is bounded above and bounded below away from $0$, as $r \to \infty$, then the number $N(r, s)$ of shifted lattice points lying inside $r\Gamma(s)$ satisfies
\begin{equation}\label{eq:boundedimprove_convex}
N(r,s) \leq r^2\area(\Gamma)-r\big(s^{-1} \tau L + s(\sigma+1/2)M \big)+ o(r) .
\end{equation}
\end{corollary}
\begin{proof}
Fix $c>1$ and assume $c^{-1}<s<c$ in the rest of the proof. 

Suppose $\sigma,\tau \leq 0$, and let $K \geq 2$. Repeat the proof of \autoref{prop:ineq_convex_shift} except with the initial requirement $L \geq 2+\sigma$ replaced by $L \geq K + \sigma$, and do not assume $s \geq 1$. The argument gives 
\[
N(r,s) \leq r^2 \area(\Gamma) -E_K rs - \tau L rs^{-1} +\sigma\tau .
\]
for all $r \geq (K+\sigma)s/L$, where 
\[
E_K = E_K(\Gamma, \sigma) = \frac{1}{2}(1+\sigma)f\big(\frac{1+\sigma}{K+\sigma}L\big) - \frac{1}{2} \sigma f(\frac{2+\sigma}{K+\sigma}L) + \sigma M .
\]
Hence
\begin{align*}
& \limsup_{r \to \infty} \sup_{s<c} \frac{1}{r}\Big( N(r,s) - r^2\area(\Gamma) + r\big(s(\sigma+1/2)M +s^{-1} \tau L \big) \Big) \\
& \leq \frac{c}{2} \Big| M - (1+\sigma)f\big(\frac{1+\sigma}{K+\sigma}L\big) + \sigma f(\frac{2+\sigma}{K+\sigma}L) \Big| .
\end{align*}
The last expression can be made arbitrarily small by choosing $K$ sufficiently large (recall $f(0)=M$), and so the left side is $\leq 0$, which proves the corollary when $\sigma,\tau \leq 0$. 

By arguing as in the proof of \autoref{cor:ineq_shift}, one handles the other three cases for $\sigma$ and $\tau$.

\end{proof}

In the next proposition we state a two-term asymptotic for lattice point counting under convex curves.
\begin{proposition}[Two-term counting estimate for convex curves]\label{th:asy_pos_convex}
Let $\sigma, \tau> -1$. If Weaker Convex Condition~\ref{cond-weak-convex} holds and $s+s^{-1} = O(1)$ then 
\begin{equation} \label{eq:asy_pos_convex}
N(r,s)=r^2\area (\Gamma)-r\big( s^{-1}(\tau + 1/2)L+s(\sigma+1/2)M\big)+ O(r^{1-2\se}) 
\end{equation}
as $r \to \infty$, where $\se = \min\{ \tfrac{1}{6}, a_1, a_2, a_3,b_1, b_2, b_3 \}$. In particular, if Convex Condition~\ref{cond:convex} holds  then \autoref{eq:asy_pos_convex} holds with $\se=\tfrac{1}{6}$. 
\end{proposition}
\autoref{th:asy_pos_convex} does not assume the intercepts $L$ and $M$ are equal, and so we modify Weaker Convex Condition~\ref{cond-weak-convex} by taking each occurrence of ``$L$'' that relates to the function $g$ and changing it to ``$M$'', and changing the $a_3$-condition to $f(x) = M+O(x^{2a_3})$.
\begin{proof}
We use the idea from \autoref{th:asy_pos}: translate and truncate the curve $r\Gamma(s)$ to reduce to an unshifted lattice problem, and then use results from Ariturk and Laugesen's paper \cite{AL17}. 

Assume $r\Gamma(s)$ does not pass through any point in the shifted lattice. This assumption will be removed in the final step of the proof. 

\smallskip Step 1 --- Translating and truncating.
Keep the notation from the proof of \autoref{th:asy_pos}, except redefine the quantities $\widetilde{\delta}$ and $\widetilde{\epsilon}$ to be
\[
\widetilde{\delta} = \big[ \widetilde{L} + 1+\sigma - rs^{-1} (L - \delta(r))  \big]^+ , \qquad \widetilde{\epsilon} = \big[ \widetilde{M} + 1+\tau - rs (M - \epsilon(r))  \big]^+ .
\]
Arguing as in Step~1 of that proof, we have
\[
0 < \widetilde{\alpha} < \lfloor \widetilde{L} \rfloor, \qquad 0 < \widetilde{\beta} < \lfloor\widetilde{M}\rfloor ,
\]
by taking $r$ large enough, and also
\[
0 \leq \widetilde{\delta}< \lfloor \widetilde{L} \rfloor-\widetilde{\alpha},\qquad 0 \leq \widetilde{\epsilon}< \lfloor \widetilde{M} \rfloor-\widetilde{\beta} .
\]

\smallskip
Step~2 --- Estimating the counting function.
Recall $F$ represents the antiderivative of $f$, defined in \autoref{eq:F_def}. Applying part (a) of \cite[Proposition~6.1]{AL17} to the curve $\sGamma$ and using the relationships between the unshifted and shifted quantities as in the proof of \autoref{th:asy_pos}, we get

\begin{align}\label{eq:two-term-asym-exact-convex}
&\Big|\sN-r^2\area (\Gamma) + r^2\Big(F\big((1+\sigma)s/r\big) + G\big((1+\tau)s^{-1}/r\big)\Big)\nonumber\\
&\hspace{1cm}+ \frac{r}{2}\Big(sf\big((1+\sigma)s/r\big) +s^{-1}g\big((1+\tau)s^{-1}/r)\big)\Big)\Big|\nonumber\\
&\leq 6r^{2/3}\Big(\int_\alpha^{L} f''(x)^{1/3} \ud x+\int_\beta^{M} g''(y)^{1/3}\ud y\Big)
+175r^{1/2}\Big(\frac{s^{-3/2}}{f''\big(L- \delta(r)\big)^{1/2}}\nonumber\\
&\hspace{1cm}+\frac{
s^{3/2}}{g''\big(M-\epsilon(r)\big)^{1/2}}\Big)+700r^{1/2}\Big(\sum_{i=0}^{l-1}\frac{s^{-3/2}}{f''(\alpha_i)^{1/2}}+\sum_{j=0}^{m-1}\frac{
s^{3/2}}{g''(\beta_j)^{1/2}}\Big)\nonumber\\
& \hspace{1cm}+\frac{1}{4} \Big(
\sum_{i=0}^{l-1}s^2|f'(\alpha_i)|+\sum_{j=0}^{m-1}s^{-2}|g'(\beta_j)| \Big) +\frac{1}{2}r\big(s^{-1}\delta(r)+s\epsilon(r)\big)+
l+m\nonumber\\
& \hspace{1cm} +\frac{1}{2}(1+\sigma)+\frac{1}{2}(1+\tau)+(1+\sigma)(1+\tau)+5
 \nonumber \\
&\hspace{1cm}+\frac{rs^{-1}g((1+\tau)/rs)-(1+\sigma)}{rsf((1+\sigma)s/r)-(1+\tau)}+\frac{rsf((1+\sigma)s/r)-(1+\tau)}{rs^{-1}g((1+\tau)/rs)-(1+\sigma)},
\end{align}
where we estimated the term involving $\shiftf''(\sL-\widetilde{\delta})^{-1/2}$ as follows. One has $\shiftf''(\sL-\widetilde{\delta}) =r^{-1}s^3 f''(z)$ where 
\[
z=r^{-1}s (\widetilde{L}-\widetilde{\delta}+1+\sigma) \leq L-\delta(r) ,
\]
and so by monotonicity of $f''$ on each subinterval of the partition (as assumed in Weaker Convex Condition \ref{cond-weak-convex}) one concludes
\[
\shiftf''(\sL-\widetilde{\delta})\geq r^{-1}s^3 \min \{ f''\big(L-\delta(r)\big),  f''\big(\alpha_0\big),\dots,f''\big(\alpha_{l-1}\big)\}.
\]
Thus the term involving $\shiftf''(\sL-\widetilde{\delta})^{-1/2}$ can be estimated by the sum of terms involving $f''(L-\delta(r))^{-1/2}$ and $f''(\alpha_i)^{-1/2}$.

The right side of \autoref{eq:two-term-asym-exact-convex} has the form $O(r^{1-2e})$, by arguing directly with $s+s^{-1}=O(1)$ and the assumptions in Weaker Convex Condition~\ref{cond-weak-convex}, and estimating the last two terms in \autoref{eq:two-term-asym-exact-convex} by
\[
\frac{rs^{-1}g((1+\tau)/rs)-(1+\sigma)}{rsf((1+\sigma)s/r)-(1+\tau)}=\frac{s^{-1}L-o(1)}{sM-o(1)} = O(1) 
\]
and similarly with $f$ and $g$ interchanged. 

\smallskip
Step~3 --- Understanding the left side of inequality \autoref{eq:two-term-asym-exact-convex}. The terms on the left of \autoref{eq:two-term-asym-exact-convex} are dealt with in the same manner as in Step~3 of \autoref{th:asy_pos}, except replacing \autoref{lemma:f-estimate} with the last assumption in Weaker Convex Condition~\ref{cond-weak-convex}, as follows. Substituting $x=(1+\sigma)s/r$ into $f(x)=M+O(x^{2a_3})$ and into $F(x)=Mx+O(x^{1+2a_3})$ gives
\begin{align*}
rsf((1+\sigma)s /r) &=rsM +O(r^{1-2a_3}),\\
r^2F((1+\sigma) s/r) &=rs (1+\sigma) M +O(r^{1-2a_3}),
\end{align*}
since $s+s^{-1}=O(1)$. One argues similarly for $g$ and $G$. Thus we have finished the proof under the assumption that $r\Gamma(s)$ passes through no lattice points. 

\smallskip
Step~4 --- Finishing the proof. Now drop the assumption that $r\Gamma(s)$ passes through no lattice points. Notice the counting function $N(r,s)$ is increasing in the $r$-variable. Fix the $r$ and $s$ values, and modify the functions $\delta(\cdot)$ and $\epsilon(\cdot)$ to be continuous at $r$. For sufficiently small $\eta>0$ we have $N(r+\eta,s)=N(r,s)$, because the $r$-variable would have to increase by some positive amount for the curve $r\Gamma(s)$ to reach any new lattice points. Since no lattice points lie on the curve $(r + \eta)\Gamma(s)$, Steps~1--3 above apply to that curve. Hence by continuity as $\eta \to 0$, the conclusion of the proposition holds also for $r\Gamma(s)$.
\end{proof}

\section{\bf Lower bound on the counting function for decreasing $\Gamma$}\label{section_lower_bound}
We need a rough lower bound on the counting function, in order to prove boundedness of the maximizing set in \autoref{thm:s_bounded_shift}. Assume the curve $\Gamma$ is strictly decreasing in the first quadrant, and has $x$- and $y$-intercepts at $L$ and $M$. The intercepts need not be equal, in the next lemma. 

\begin{lemma}[Rough lower bound for decreasing curve]\label{lemma:lower_bound}
The number $N(r, s)$ of shifted lattice points lying inside $r\Gamma(s)$ satisfies
\begin{equation}\label{eq:s_bounded_n_lower_bound} 
N(r,s)\geq r^2\area(\Gamma)-r\big(s^{-1}(1+\tau)L+s(1+\sigma)M\big), \qquad r,s>0.
\end{equation}
\end{lemma}
\begin{proof}
We split the proof into two cases, and later rescale to handle the general curve. Write $N$ for the number of shifted lattice points under $\Gamma$. 

Case I: The point $(1+\sigma,1+\tau)$ lies outside the curve $\Gamma$, and so $N=0$. Then the rectangles with vertices $(0,0),(L,0),(L,1+\tau),(0,1+\tau)$ and $(0,0),(1+\sigma,0),(1+\sigma,M),(0,M)$ cover $\Gamma$ since the curve is decreasing, and so by comparing areas one has 
\begin{equation} \label{eq:roughlower}
N +(1+\tau)L + (1+\sigma)M\geq \area(\Gamma).
\end{equation}

Case II: The point $(1+\sigma,1+\tau)$ lies inside the curve. We shift the origin to $\sO=(1+\sigma, 1+\tau)$ and draw new axes, denoting the $x$- and $y$-intercepts on the new axes by $\sL$ and $\sM$; see \autoref{fig:s_bound_pos_shift}. The part of $\Gamma$ lying in the new first quadrant is $\sGamma$.
\begin{figure}
\includegraphics[scale=0.4]{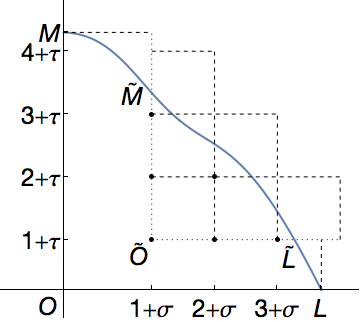}
\caption{\label{fig:s_bound_pos_shift} Decreasing curve $\Gamma$ enclosing positive integer lattice points shifted by amount $(\sigma,\tau)=(0.4, -0.2)$. We shift the origin by $1+\sigma, 1+\tau$, obtaining a new origin $\sO$, with $\sL$ and $\sM$ being the new $x$- and $y$-intercepts. The lattice point count equals the area of the squares, as used in proving \autoref{lemma:lower_bound}.}
\end{figure}
Each lattice point corresponds to a square whose lower left vertex sits at that point. These squares cover $\sGamma$ since the curve is strictly decreasing. The remaining area under $\Gamma$ is covered by the two rectangles described in Case~I. The sum of the areas of the squares and rectangles must exceed the area under $\Gamma$, and so \autoref{eq:roughlower} holds once again. 

To complete the proof, simply replace the curve $\Gamma$ with $r\Gamma(s)$, meaning that in \autoref{eq:roughlower} we replace $N$, $L$, $M$ with $N(r,s)$, $rs^{-1}L$, $rsM$ respectively. The lemma follows.
\end{proof}

\section{\bf Proof of \autoref{thm:s_bounded_shift}}\label{section_proof_bounded}
We prove the theorem in two parts: first for concave curves, and then for convex curves. When $\Gamma$ is concave, we will utilize the bound on $S(r)$ in \autoref{lemma:bound_shift} and the two-term upper bound on the counting function in \autoref{prop:ineq_shift}, along with the improved upper bound in \autoref{cor:ineq_shift} and the rough lower bound on the counting function in \autoref{lemma:lower_bound}. 

Recall the intercepts are assumed equal ($L=M$) in this theorem. 

\subsection*{Part 1: $\Gamma$ is concave and Parameter Assumption~\autoref{pa:concave} holds}\ 
The proof has two steps. Step~1 shows $S(r)$ is bounded above and below away from $0$, for large $r$. Step~2 uses this boundedness to improve the asymptotic bound on $S(r)$, revealing that it depends only on $\sigma$ and $\tau$ and not the curve $\Gamma$. 

\smallskip
\noindent Step~1. Take $s\in S(r)$ and suppose $r \geq (2+\sigma+\tau)/L$. Then \autoref{lemma:bound_shift} says $s\leq rL/(1+\sigma)$, so that 
\[
r \geq \frac{(1+\sigma)s}{L} \geq \frac{(1-\sigma^-)s}{L}.
\]
If $s\geq 1$ then \autoref{prop:ineq_shift} implies
\begin{equation*}\label{eq:s_bounded_upper_bound}
N(r,s)\leq r^2 \area(\Gamma)-C_1rs + \sigma^- \tau^- .
\end{equation*}
Parameter Assumption~\autoref{pa:concave} guarantees here that $C_1>0$.

The lower bound in \autoref{lemma:lower_bound} with ``$s = 1$'' says
\begin{equation}\label{eq:s-1-lower-bound}
N(r,1) \geq r^2 \area(\Gamma)-(2+\sigma+\tau)Lr.
\end{equation}
Since $s\in S(r)$ is a maximizing value, one has $N(r,s) \geq N(r,1)$, and so the preceding two inequalities give 
\[
s \leq \frac{(2+\sigma+\tau)L}{C_1} + \frac{\sigma^- \tau^- L}{(2+\sigma+\tau)C_1}
\]
when $r \geq (2+\sigma+\tau)/L$ and $s \geq 1$. Thus $S(r)$ is bounded above for all large $r$. 

Similarly if $s\in S(r)$ then $s^{-1}$ is bounded above, by interchanging the roles of the horizontal and vertical axes in the argument above. Thus the set $S(r)$ is bounded below away from $0$, for large $r$.

\smallskip
Step~2. The number 
\[
\overline{s} = \limsup_{s \in S(r), r \to \infty} s
\]
is finite and positive by Step~1. Combining the inequality $N(r,s) \geq N(r,1)$ with estimate \autoref{eq:s-1-lower-bound} and \autoref{cor:ineq_shift} (which relies on the boundedness of $S(r)$) we obtain  
\[
(\sigma+1/2) \overline{s}^2 - (2+\sigma+\tau) \overline{s} + \tau \leq 0 
\]
after letting $r \to \infty$. Notice $\sigma+1/2>0$ by hypothesis in \autoref{thm:s_bounded_shift}. Hence $\overline{s}$ is bounded above by the larger root of the quadratic; that is,
\[
\overline{s} \leq B(\sigma,\tau)= \frac{2+\sigma+\tau + \sqrt{ (2+\sigma + \tau)^2 - 4(\sigma+1/2)\tau}}{2(\sigma+1/2)} .
\]
Similarly $\limsup_{r\to \infty} s^{-1} \leq B(\tau,\sigma)$, by interchanging the roles of the axes. The proof of \autoref{thm:s_bounded_shift} is complete, in the concave case. 
 
\subsection*{Part 2: $\Gamma$ is convex and Parameter Assumption~\autoref{pa:convex} holds}\ 

Take $s \in S(r)$ and suppose $r$ satisfies \eqref{eq:improvedcondition}, recalling there that $\mu_f(\sigma)$ and $\mu_g(\tau)$ are positive by Parameter Assumption~\autoref{pa:convex}. Now proceed as in Part~1 of the proof, simply replacing \autoref{lemma:bound_shift}, \autoref{prop:ineq_shift} and \autoref{cor:ineq_shift} with \autoref{lemma:r_bound_negative}, \autoref{prop:ineq_convex_shift} and \autoref{cor:ineq_convex_shift}, respectively.

\section{\bf Proof of \autoref{th:S_limit_shift_general}} \label{section_proof}

Recall the intercepts are equal, $L=M$, in this theorem. 

The optimal stretch parameters are bounded above and bounded below away from $0$ as $r \to \infty$, by \autoref{thm:s_bounded_shift}. (It suffices to use the curve-dependent bound from Step~1 of that proof; we do not need the curve-independent bound $B(\sigma,\tau)$ from Step~2.) 

Hence by \autoref{th:asy_pos} (if $\Gamma$ is concave) or \autoref{th:asy_pos_convex} (if $\Gamma$ is convex), 
\begin{align}
N(r,s) =r^2\area (\Gamma)-rL\big(s^{-1}(\tau + 1/2)+s(\sigma+1/2)\big)+ O(r^{1-2\se}) 
\label{eq:estimateO}
\end{align}
when $s \in S(r)$; this estimate holds also when $s>0$ is any fixed value. Thus for $s \in S(r)$ and $s^*=\sqrt{(\tau + 1/2)/(\sigma+1/2)}$ we have
\begin{align*}
N(r,s)&\leq r^2\area (\Gamma)-rL\big(s^{-1}(\tau+ 1/2)+s(\sigma+1/2)\big)+ O(r^{1-2\se}), \\
N(r,s^*) & \geq r^2\area (\Gamma)-2rL\sqrt{(\tau+1/2)(\sigma + 1/2)} + O(r^{1-2\se}),
\end{align*}
as $r \to \infty$. Notice $N(r, s^*)\leq N(r,s)$ because $s \in S(r)$ is a maximizing value, and so
\begin{equation} \label{eq:sinverse}
 s^{-1}(\tau + 1/2)+s(\sigma+1/2) \leq 2\sqrt{(\tau + 1/2)(\sigma+1/2)}+  O(r^{-2\se}) .
\end{equation}
Therefore $s =s^*+ O(r^{-\se})$, by \autoref{le:squarecompletion} below with $a = \tau + 1/2, b = \sigma+1/2$. 

For the final statement of the theorem, when $s\in S(r)$ one has 
\[
2\sqrt{(\tau + 1/2)(\sigma+1/2)} \leq s^{-1}(\tau + 1/2)+s(\sigma+1/2)\leq 2\sqrt{(\tau + 1/2)(\sigma+1/2)}+  O(r^{-2\se})
\]
by the arithmetic--geometric mean inequality and \autoref{eq:sinverse}. Multiplying by $rL$ and substituting into \autoref{eq:estimateO} gives the asymptotic formula \autoref{eq:maximalasymptotic}. 

\begin{lemma} \label{le:squarecompletion}
When $a, b, s >0$ and $0 \leq t \leq \sqrt{ab}$, 
\[
s^{-1}a + sb \leq 2\sqrt{ab} +t \quad \Longrightarrow \quad \big| s - \sqrt{a/b} \big| \leq \frac{ 3(ab)^{1/4}}{ b}\sqrt{t} .
\]
\end{lemma}
\begin{proof} By taking the square root on both sides of the inequality 
\[
( (s^{-1}a)^{1/2}-(sb)^{1/2})^2 = s^{-1}a + sb - 2\sqrt{ab} \leq t 
\]
and then using that the number $(ab)^{1/4}$ lies between $(s^{-1}a)^{1/2}$ and $(sb)^{1/2}$ (because it is their geometric mean), we find
\[
 | (ab)^{1/4} - (sb)^{1/2} | \leq t^{1/2} .
\]
Hence $(ab)^{1/4} - t^{1/2} \leq (sb)^{1/2} \leq (ab)^{1/4} + t^{1/2}$. Squaring and using that
$t \leq (ab)^{1/4}t^{1/2}$ (when $t \leq \sqrt{ab}$) proves the lemma.  
\end{proof}

\section{\bf Proof of \autoref{th:neg_shift} and \autoref{prop:neg_ellipse}}

\subsection*{Proof of \autoref{th:neg_shift}}
\label{section:proof_neg}
Fix $\sigma \in (-1,0)$ and $\tau>-1$. Since $0<1+\sigma<1$, we may choose $m \in \N$ large enough that
\begin{equation} \label{eq:countereg}
(1+\sigma)^{2m} < \frac{1}{2m+1} .
\end{equation}
Defining $\phi(x) = 1-x^{2m}$ for $0 \leq x \leq 1$, one checks $\phi(1+\sigma) > \text{area under graph of $\phi$}$. Thus one may choose $0<\delta<1$ small enough that the function
\[
f(x) = 1-\delta x^2 - (1-\delta) x^{2m} , \qquad 0 \leq x \leq 1 ,
\]
satisfies
\[
f(1+\sigma) > \text{area under graph of $f$} .
\]
Observe $f$ is smooth and strictly decreasing, with $f'' < 0$ on $[0,1]$, so that its graph $\Gamma$ is concave. The inverse function $g$ satisfies the same conditions. 

The curve $r\Gamma(r)$ is the graph of $r^2f(x)$ for $0 \leq x \leq 1$. This curve contains only the first column of shifted lattice points (the points with $x$-coordinate $1+\sigma$), and so 
\begin{align*}
N(r,r) 
& = \lfloor r^2 f(1+\sigma) - \tau \rfloor \\
& \geq r^2 f(1+\sigma) - \tau - 1.
\end{align*}

Now fix $0<\epsilon<1$. If $s \in [r^{\epsilon-1},r^{1-\epsilon}]$ then $s + s^{-1} = O(r^{1-\epsilon})$, and so \autoref{th:asy_pos} with $q=1-\epsilon$ and $L=M=1$ gives that
\begin{align*}
N(r,s) 
& = r^2 \area(\Gamma) - r\big(s(\sigma+1/2) + s^{-1}(\tau + 1/2)\big) + O(r^{2-3\epsilon/2}) \\
& = r^2 \area(\Gamma) + o(r^2) .
\end{align*}
Since $\area(\Gamma) < f(1+\sigma)$, we conclude that for all large $r$,
\[
N(r,s) < N(r,r)
\]
and so $s \notin S(r)$, which proves the theorem. 

\subsection*{Proof of \autoref{prop:neg_ellipse}}
\label{section:proof_prop}

By symmetry, we may suppose $\sigma \leq -2/5$.

The argument is the same as for \autoref{th:neg_shift}, except now the curve is a quarter circle, described by $f(x) = \sqrt{1-x^2}$. The only point to check in the proof is that 
\[
f(1+\sigma) > \area(\Gamma) 
\]
when $-1<\sigma\leq-2/5$, which reduces to the fact that $4/5 > \pi/4$.

\section{\bf Numerical examples, and conjectures for triangles $(p=1)$ }\label{sec:numerics}
 
\autoref{fig:optimal_s_p_2}(a) illustrates the convergence of $s \in S(r)$ to $s^*$, when $\Gamma$ is a quarter circle and the shifts are positive. The convergence is erratic, while still obeying the decay rate $O(r^{-1/6})$ as promised by \autoref{th:S_limit_shift}. \autoref{fig:optimal_s_p_2}(b) shows the degeneration that can occur when the shifts are negative, as explained in \autoref{prop:neg_ellipse}. 

\begin{figure}
\includegraphics[scale=0.4]{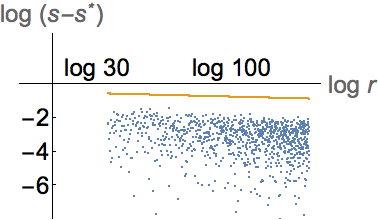} \hspace{1cm}
\includegraphics[scale=.4]{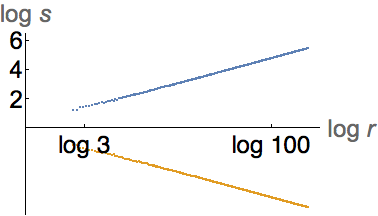}
\caption{\label{fig:optimal_s_p_2}Maximizing $s$-values for the number of lattice points in the $2$-ellipse. (a) Left figure: positive shift $\sigma=1,\tau = 3$. The plot shows $\log(\sup S(r)-s^*)$ versus $\log r$. The line $-1/6 \log r$ indicates the guaranteed convergence rate in \autoref{th:S_limit_shift}. (b) Right figure: negative shifts $\sigma= \tau = -2/5$. The plot shows $\log(\sup S(r))$ and $\log(\inf S(r))$ versus $\log r$. Linear fitting gives $\log \sup S(r) \simeq 0.982\log r +0.254$, which is consistent with the growth rate $s \gtrsim r^{1-\epsilon}$ proved in \autoref{prop:neg_ellipse}. In both plots, the $r$-values are multiples of $\sqrt{3}/10$, an irrational number chosen in the hope of exhibiting generic behavior.}
\end{figure}

Quite different behavior occurs when $\Gamma$ is a straight line with slope $-1$, in other words, when the curve is the $1$-ellipse described by $f(x)=1-x$, which is not covered by our results in \autoref{ex:p_shift}. Here $N(r,s)$ counts the shifted lattice points inside the right triangle with vertices at $(r/s,0), (0,rs)$ and the origin. \autoref{thm:s_bounded_shift} insures the maximizing set $S(r)$ is bounded above and below, being contained in $\big[B(\tau,\sigma)^{-1}-\e,B(\sigma,\tau)+\e \big]$ for all large $r$. This boundedness depends on Parameter Assumption~\autoref{pa:concave} holding, which in this case says
\[
(2-\max(\sigma^-,\tau^-))(1-2\sigma^- - 2\tau^-) > 1 .
\]
In particular, $S(r)$ is bounded for the $1$-ellipse if $\sigma=\tau>-0.117$. \autoref{th:S_limit_shift} for convergence of $S(r)$ does not apply, though, to the $1$-ellipse. 

The numerical plots in \autoref{fig:optimal_s_p_1} suggest $S(r)$ might not converge, and might instead cluster at many different heights. Are those heights determined by a number theoretic property of some kind? (Such behavior would be particularly interesting when the shifts are $\sigma=\tau=-1/2$, since those shifted lattice points correspond to energy levels of harmonic oscillators in $2$-dimensions, as explained in the next section.) For a more detailed discussion and precise conjecture on this open problem for $p=1$ in the unshifted case, see our work in \cite[Section~9]{Lau_Liu17} and the partial results of Marshall and Steinerberger \cite{marshall_steinerberger}.

The numerical method that generated the figures is described in \cite[Section~9]{Lau_Liu17} for $p=1$. It adapts easily to handle other values of $p$, in particular $p=2$ (the circle), and the code is available in \cite[Appendix~B]{thesis}. 

\begin{figure}
\includegraphics[scale=0.35]{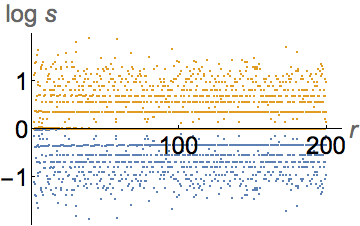} \hspace{.35cm}
\includegraphics[scale=0.35]{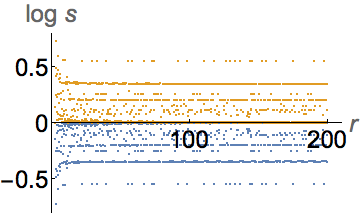} \hspace{.35cm}
\includegraphics[scale=0.35]{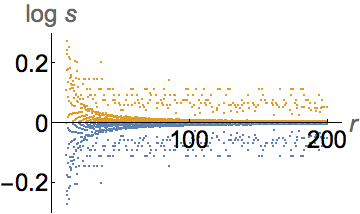}
\caption{\label{fig:optimal_s_p_1}Maximizing $s$-values for the number of lattice points in the $1$-ellipse (that is, the right triangle). The upper plots show $\log \sup S(r)$ versus $r$ and the lower plots are $\log \inf S(r)$ versus $r$. The figure on the left is for shift parameters $\sigma=\tau = -1/2$, which corresponds to counting eigenvalues of the harmonic oscillator (\autoref{section_spectral}). The middle figure has $\sigma= \tau =0$, and the figure on the right has $\sigma= \tau =4$. Notice the optimal stretch parameters are bounded in a narrower and narrower band as the shift parameters increase.}
\end{figure}

\section{\bf Future directions --- optimal quantum oscillators}\label{section_spectral}

\subsection*{Literature on spectral minimization}
Antunes and Freitas \cite{AF13} investigated the problem of maximizing the number of first-quadrant lattice points in ellipses with fixed area, and showed that optimal ellipses must approach a circle as the radius approaches infinity. In terms of eigenvalues of the Dirichlet Laplacian on rectangles having fixed area, their result says that the rectangle minimizing the $n$-th eigenvalue must approach a square as $n \to \infty$. Their intuition is that high eigenvalues should be asymptotically minimal for the ``most symmetrical'' domain. 

Besides Antunes and Freitas's work on eigenvalue minimization \cite{AF12,AF13,AF16,f16}, we mention that van den Berg and Gittins \cite{BBG16b} showed the cube is asymptotically minimal in $3$-dimensions as $n \to \infty$, while Gittins and Larson \cite{GL17} handle all dimensions $\geq 2$. Van den Berg, Bucur and Gittins \cite{BBG16a} proved an analogous asymptotic maximization result for eigenvalues of the Neumann Laplacian on rectangles. Bucur and Freitas \cite{BF13} showed for general domains in $2$ dimensions that eigenvalue minimizing regions become circular in the limit, under a perimeter normalization. Larson \cite{Larson} shows among convex domains that the disk asymptotically minimizes the Riesz means of the  eigenvalues, for Riesz exponents $\geq 3/2$. Eigenvalue minimizing domains have been studied numerically by Oudet \cite{Oud04}, Antunes and Freitas \cite{AF13}, and Antunes and Oudet \cite{Freitas_oudet}, \cite[Chapter 11]{H17}. Incidentally, Colbois and El Soufi \cite[Corollary~2.2]{colbois_soufi} proved subadditivity of $n \mapsto \lambda_n^*$ (the minimal value of the $n$-eigenvalue), from which it follows that the the famous P\'olya conjecture $\lambda_n \geq 4\pi n /\area$ in $2$-dimensions would be a corollary of the conjecture that the eigenvalue minimizing domain approaches a disk as $n \to \infty$. 

A different way of extending the work of Antunes and Freitas is to investigate lattice point counting inside more general curves, not just ellipses. Laugesen and Liu \cite{Lau_Liu17} and Laugesen and Ariturk \cite{AL17} showed this can be done for $p$-ellipses with $p \neq 1$, and for more general concave and convex curves in the first quadrant too. Marshall \cite{marshall} has extended the results to strongly convex domains in all dimensions, using somewhat different methods. For more on the literature see \cite{AL17}.

\subsection*{Spectral application of \autoref{prop:neg_ellipse}}
This paper sheds new light on the rectangular result of Antunes and Freitas. Consider the family of rectangles defined by
\[
[0,2\pi s^{-1}] \times [0,2\pi s]
\]
for $s>0$. The ``even--even'' eigenfunctions (that are symmetrical with respect to the two axes through the center) have the form
\[
u = \sin \big( s(j-1/2)x \big) \sin \big( s^{-1}(k-1/2)y \big) 
\]
with corresponding eigenvalues
\[
\lambda = \Big( s(j-\frac{1}{2}) \Big)^{\! 2} + \Big( s^{-1}(k-\frac{1}{2}) \Big)^{\! 2} ,
\]
for $j,k \geq 1$. These even--even eigenvalues have counting function 
\begin{align*}
& \# \{ \lambda \leq r^2 \} \\
& = \text{number of points in the shifted lattice $(\N - 1/2) \times (\N-1/2)$ lying} \\
& \hspace{3cm} \text{inside or on the ellipse $(sx)^2+(s^{-1}y)^2 \leq r^2$} \\
& = N(r,s)
\end{align*}
where the shift parameters are $\sigma=\tau=-1/2$ and the curve $\Gamma$ is the quarter-circle.

\autoref{prop:neg_ellipse} says the set $S(r)$ of $s$-values that maximize the counting function does \emph{not} approach $1$ as $r \to \infty$. Instead, the maximizing $s$-values approach $0$ or $\infty$. Thus the even--even symmetry class of eigenvalues on the rectangle behaves quite differently from the full collection of eigenvalues studied by Antunes and Freitas. The asymptotically optimal rectangle for maximizing the counting function as $r \to \infty$ (or equivalently, minimizing the $n$-th eigenvalue as $n \to \infty$) is not the square but rather the degenerate rectangle. 

\subsection*{Open problem for harmonic oscillators}
A quantum analogue of the Antunes--Freitas theorem for rectangles would be to minimize the $n$-th energy level among the following family of harmonic oscillators. For each $s>0$, consider 
\begin{equation} \label{eq:harmonicosc}
-\Delta u + \frac{1}{4} \big( (sx)^2 + (s^{-1}y)^2 \big) u = E u , \qquad x,y \in \R ,
\end{equation}
with boundary condition $u \to 0$ as $|(x,y)| \to \infty$. Write $s_n$ for an $s$-value that minimizes the $n$-th eigenvalue $E_n$. By analogy with Antunes and Freitas's theorem for Dirichlet rectangles, one might conjecture that $s_n \to 1$ as $n \to \infty$. In fact, the behavior is quite different, as we now explain. 

Let us translate the harmonic oscillator problem into a shifted lattice point counting problem. The $1$-dimensional oscillator equation $-u^{\prime \prime} + \frac{1}{4} x^2 u = E u$ has eigenvalues $E=j-1/2$ for $j=1,2,3,.\dots$. By separating variables and rescaling, one finds that equation \autoref{eq:harmonicosc} has spectrum 
\[
\{ E_n \} = \{ s(j-1/2) + s^{-1}(k-1/2) : j,k = 1,2,3,\dots \} .
\] 
Hence the number of harmonic oscillator eigenvalues less than or equal to $r$ equals the number of points in the shifted lattice $(\N - 1/2) \times (\N - 1/2)$ lying below the straight line $sx+s^{-1}y=r$, which is given by our counting function $N(r,s)$ where $\Gamma$ is the straight line $y=1-x$ (the $1$-ellipse) and the shift parameters are $\sigma=\tau=-1/2$. To minimize the eigenvalues we should maximize the counting function. 

The numerical evidence in the left part of \autoref{fig:optimal_s_p_1} suggests that the $s$-values maximizing the counting function $N(r,s)$ do not converge to $1$ as $r \to \infty$. Rather, the optimal $s$-values seem to cluster at various heights. (For a precise such clustering conjecture in the unshifted case, see \cite[Section~9]{Lau_Liu17}.) Thus the family of harmonic oscillators exhibits strikingly different spectral behavior from the family of Dirichlet rectangles. 

\subsection*{Interpolating family of Schr\"{o}dinger potentials}
The family of Schr\"{o}dinger potentials $|sx|^q + |s^{-1}y|^q$, where $2<q<\infty$ and $s>0$, interpolates between the harmonic oscillator ($q=2$) and the infinite potential well ($q=\infty$) that corresponds to the Dirichlet Laplacian on a rectangular domain. We conjecture that when $2<q<\infty$, the set $S(r)$ of values maximizing the eigenvalue counting function will converge to $1$ as $r \to \infty$. This conjecture would provide a $1$-parameter family of quantum oscillators for which the analogue of the Antunes--Freitas theorem holds true, with the family terminating in an exceptional endpoint case: the harmonic oscillator. 

The difficulty is that the eigenvalues of the $1$-dimensional oscillator with potential $|x|^q$ do not grow at a precisely regular rate. Hence to tackle the conjecture, one will need to extend the current paper from shifted lattices, where each row and column of the lattice is translated by the same amount, and find a way to handle \emph{deformed} lattices, where the amount of translation varies with the rows and columns. This challenge remains for the future.

\section*{Acknowledgments}
This research was supported by grants from the Simons 
Foundation (\#429422 to Richard Laugesen) and the University 
of Illinois Research Board (award RB17002). 
The material in this paper forms part of Shiya Liu's Ph.D. dissertation at the University of Illinois, Urbana--Champaign \cite{thesis}.


\newpage

\begin{thebibliography}{99}

\bibitem{AF12} P. R. S. Antunes and P. Freitas. \emph{Numerical optimization of low eigenvalues of the Dirichlet and Neumann Laplacians.} J. Optim. Theory Appl. 154 (2012), 235--257.

\bibitem{AF13}  P. R. S. Antunes and P. Freitas. \emph{Optimal spectral rectangles and lattice ellipses.} Proc. R. Soc. Lond. Ser. A Math. Phys. Eng. Sci. 469 (2013), no.~2150, 20120492, 15 pp.

\bibitem{AF16} P. R. S. Antunes and P. Freitas. \emph{Optimisation of eigenvalues of the Dirichlet Laplacian with a surface area restriction.} Appl. Math. Optim. 73 (2016), no.~2, 313--328.

\bibitem{Freitas_oudet} P. R. S. Antunes and \'{E}. Oudet. \emph{Numerical minimization of Dirichlet--Laplacian eigenvalues of four-dimensional geometries.} SIAM J. Sci. Comput., to appear.

\bibitem{AL17} S. Ariturk and R. S. Laugesen. \emph{Optimal stretching for lattice points under convex curves.} Port. Math., to appear. \arxiv{1701.03217}

\bibitem{BBG16a} M. van den Berg, D. Bucur and K. Gittins. \emph{Maximizing Neumann eigenvalues on rectangles.} Bull. Lond. Math. Soc. 48 (2016), no.~5, 877--894.

\bibitem{BBG16b} M. van den Berg and K. Gittins. \emph{Minimising Dirichlet eigenvalues on cuboids of unit measure.} Mathematika 63 (2017), no.~2, 469--482.

\bibitem{BF13} D.~Bucur and P.~Freitas. \emph{Asymptotic behaviour of optimal spectral planar domains with fixed perimeter.} J. Math. Phys. 54 (2013), no.~5, 053504.

\bibitem{colbois_soufi} B. Colbois and A. El Soufi. \emph{Extremal eigenvalues of the Laplacian on Euclidean domains and closed surfaces.} Math. Z. 278 (2014), 529--549. 


\bibitem{f16} P. Freitas. \emph{Asymptotic behaviour of extremal averages of Laplacian eigenvalues.} J. Stat. Phys. 167 (2017), no. 6,  1511--1518.  

\bibitem{GL17} K. Gittins and S. Larson \emph{Asymptotic behaviour of cuboids optimising Laplacian eigenvalues} 
\arxiv{1703.10249} 

\bibitem{H17} A. Henrot, ed. Shape Optimization and Spectral Theory. De Gruyter Open, to appear, 2017.

\bibitem{Hux96} M.~N. Huxley. Area, Lattice Points, and Exponential Sums. London Mathematical Society Monographs. New Series, vol.~13, The Clarendon Press, Oxford University Press, New York, 1996, Oxford Science Publications. 

\bibitem{Hux03} M.~N. Huxley. \emph{Exponential sums and lattice points. III.} 
Proc. London Math. Soc. (3) 87 (2003), 591--609.


\bibitem{kratzel00} E. Kr\"atzel. \emph{Analytische Funktionen in der Zahlentheorie.} Teubner--Texte zur Mathematik, 139. B. G. Teubner, Stuttgart, 2000. 288 pp.

\bibitem{kratzel04} E. Kr{\"a}tzel. \emph{Lattice points in planar convex domains.} Monatsh. Math. 143 (2004), 145--162. 

\bibitem{Larson} S. Larson. \emph{Asymptotic shape optimization for Riesz means of the Dirichlet Laplacian over convex domains.} \arxiv{1611.05680}

\bibitem{Lau_Liu17} R.~S.~Laugesen and S.~Liu. \emph{Optimal stretching for lattice points and eigenvalues.} Ark. Mat., submitted, \link.

\bibitem{thesis} S.~Liu. \emph{Asymptotically optimal shapes for counting lattice points and eigenvalues.} Ph.D. dissertation, University of Illinois, Urbana--Champaign, 2017.

\bibitem{marshall} N.~F. Marshall. \emph{Stretching convex domains to capture many lattice points.} \arxiv{1707.00682}.

\bibitem{marshall_steinerberger} N.~F. Marshall and S. Steinerberger. \emph{Triangles capturing many lattice points}. \arxiv{1706.04170}.

\bibitem{Oud04} \'{E}. Oudet. \emph{Numerical minimization of eigenmodes of a membrane with respect to the domain.} ESAIM Control Optim. Calc. Var. 10 (2004), 315--330. 


\end{thebibliography}
\end{document}